\documentclass[12pt]{amsart}
\usepackage{amssymb}
\usepackage[all]{xy}

\newtheorem{theorem}{Theorem}[section]
\newtheorem{definition}[theorem]{Definition}
\newtheorem{proposition}[theorem]{Proposition}
\newtheorem{lemma}[theorem]{Lemma}

\begin{document}

\title{Half-liberated manifolds, and their quantum isometries}

\author{Teodor Banica}
\address{T.B.: Department of Mathematics, Cergy-Pontoise University, 95000 Cergy-Pontoise, France. {\tt teodor.banica@u-cergy.fr}}


\begin{abstract}
We discuss the half-liberation operation $X\to X^*$, for the algebraic submanifolds of the unit sphere, $X\subset S^{N-1}_\mathbb C$. There are several ways of constructing this correspondence, and we take them into account. Our main results concern the computation of the affine quantum isometry group $G^+(X^*)$, for the sphere itself.
\end{abstract}

\maketitle

\section*{Introduction}

The notion of noncommutative space goes back to an old theorem of Gelfand, which states that any commutative $C^*$-algebra must be of the form $C(X)$, for a certain compact space $X$. One can therefore define the category of ``noncommutative compact spaces'' to be the category of $C^*$-algebras, with the arrows reversed. The category of usual compact spaces embeds then covariantly into this category, via $X\to C(X)$.

We will be interested here in noncommutative analogues of the compact algebraic manifolds $X\subset\mathbb C^N$. These are by definition the duals of the universal $C^*$-algebras defined with generators $z_1,\ldots,z_N$, subject to (noncommutative) polynomial relations:
$$C(X)=C^*\left(z_1,\ldots,z_N\Big|P_i(z_1,\ldots,z_N)=0\right)$$

The Gelfand theorem tells us that this construction covers all the compact algebraic manifolds $X\subset\mathbb C^N$. In general, the axiomatization of the algebras on the right is quite a tricky problem. Instead of getting into details here, let us just say that the family of noncommutative polynomials $\{P_i\}$ must be by definition such that the biggest $C^*$-norm on the universal $*$-algebra $<z_1,\ldots,z_N|P_i(z_1\ldots,z_N)=0>$ is bounded.

The compact quantum Lie groups, axiomatized by Woronowicz in \cite{wo1}, \cite{wo2}, and their homogeneous spaces, provide some key examples of such manifolds. Technically speaking, one problem with such quantum groups is that they lack an analogue of a Lie algebra. As explained in \cite{wo1}, \cite{wo2}, one solution to this issue comes from the intensive use of representation theory, in order to overcome the lack of geometric techniques.

The aim of this paper is to use some quantum group ideas, coming from representation theory, in the complex manifold setting. Let $X$ be as above, and consider its classical version $X_{class}\subset\mathbb C^N$, obtained by dividing the algebra $C(X)$ by its commutator ideal:
$$C(X_{class})=C^*_{comm}\left(z_1,\ldots,z_N\Big|P_i(z_1,\ldots,z_N)=0\right)$$

We can think then of $X$ as being a ``liberation'' of $X_{class}$, and the problem is that of understanding how the correspondence $X_{class}\to X$ can appear.

This latter question was solved in the quantum group case in \cite{bsp}, by using some inspiration from Wang's papers \cite{wa1}, \cite{wa2}, from the Weingarten formula \cite{bco}, \cite{csn}, \cite{wei}, and from free probability theory \cite{bpa}, \cite{spe}, \cite{vdn}. Among the findings there, and from the related papers \cite{bve}, \cite{bdd}, \cite{bdu}, is the fact that, for liberation purposes, the usual commutation relations $ab=ba$ can be succesfully replaced by the half-commutation relations $abc=cba$. This is actually a quite non-trivial phenomenon, which comes from the fact that the half-commutation relations $abc=cba$ have a deep categorical meaning. See \cite{bsp}.

As explained in \cite{bdd}, \cite{bdu}, there are several possible ways of half-liberating a manifold $X\subset S^{N-1}_\mathbb C$, and we will take this into account. We will show here that, under suitable assumptions on $X\subset S^{N-1}_\mathbb C$, we have a half-liberation diagram for it, as follows:
$$\xymatrix@R=15mm@C=15mm{
X\ar[r]&X^{**}\ar[r]&X^*\\
X^-\ar[r]\ar[u]&X^\circ\ar[u]\ar[r]&X^\#\ar[u]}$$

Our main results will concern the sphere $X=S^{N-1}_\mathbb C$ itself. More specifically, we will be interested in computing the quantum isometry groups of its various half-liberations. Our approach here will be based on the affine quantum isometry group formalism \cite{chi}, \cite{gos}, \cite{hua}, \cite{qsa}, with various technical ingredients from \cite{ba1}, \cite{bve}, \cite{bdu}, \cite{rau}. We will prove that the affine quantum isometry groups of the $6$ half-liberated spheres are as follows:
$$\xymatrix@R=14mm@C=11mm{
S^{N-1}_\mathbb C\ar[r]&S^{N-1}_{\mathbb C,**}\ar[r]&S^{N-1}_{\mathbb C,*}\\
\mathbb TS^{N-1}_\mathbb R\ar[r]\ar[u]&S^{N-1}_{\mathbb C,\circ}\ar[u]\ar[r]&S^{N-1}_{\mathbb C,\#}\ar[u]}\quad
\xymatrix@R=10mm@C=10mm{\\ \longrightarrow}
\quad
\xymatrix@R=14.8mm@C=14mm{
U_N\ar[r]&U_N^{**}\ar[r]&U_N^*\\
\mathbb TO_N\ar[r]\ar[u]&U_N^\circ\ar[u]\ar[r]&U_N^\#\ar[u]}$$

In other words, our result will state that, for the sphere $X=S^{N-1}_\mathbb C$ itself, the quantum isometry groups of the half-liberations are the half-liberations of the usual isometry group. This could be thought of as being related to the various rigidity results in \cite{bhg}, \cite{gjo}.

The paper is organized as follows: in 1-2 we discuss the half-liberation operation for the complex sphere itself, in 3-4 we study the associated quantum isometry groups, and in 5-6 we discuss the case of more general algebraic manifolds $X\subset S^{N-1}_\mathbb C$.

\medskip

\noindent {\bf Acknowledgements.} I would like to thank Julien Bichon for several useful discussions, and the anonymous referee for a careful reading of the manuscript. This work was partly supported by the NCN grant 2012/06/M/ST1/00169.

\section{Noncommutative spheres}

According to \cite{ba2}, \cite{bgo}, which were based on the previous work of Wang in \cite{wa1}, the free analogue of the complex unit sphere $S^{N-1}_\mathbb C$ is constructed as follows:

\begin{definition}
Associated to any $N\in\mathbb N$ is the universal $C^*$-algebra
$$C(S^{N-1}_{\mathbb C,+})=C^*\left(z_1,\ldots,z_N\Big|\sum_iz_iz_i^*=\sum_iz_i^*z_i=1\right)$$
whose abstract spectrum $S^{N-1}_{\mathbb C,+}$ is called free analogue of $S^{N-1}_\mathbb C$.
\end{definition}

Observe that the classical version of $S^{N-1}_{\mathbb C,+}$, obtained by assuming in addition that the standard coordinates $z_i$ and their adjoints $z_i^*$ commute, is the usual sphere $S^{N-1}_\mathbb C$. This follows indeed from the Stone-Weierstrass and Gelfand theorems. See \cite{ba2}, \cite{bgo}.

We will be interested in what follows in various half-liberated analogues of $S^{N-1}_\mathbb C$. We have the following constructions here, which go back to the work in \cite{ba2}:

\begin{definition}
We have the following subspheres of $S^{N-1}_{\mathbb C,+}$:
\begin{enumerate}
\item $S^{N-1}_{\mathbb C,*}$: obtained via the relations $ab^*c=cb^*a$, with $a,b,c\in\{z_i\}$.

\item $S^{N-1}_{\mathbb C,**}$: obtained via the relations $abc=cba$, with $a,b,c\in\{z_i,z_i^*\}$.

\item $S^{N-1}_{\mathbb C,\#}$: obtained via the relations $ab^*=ba^*,a^*b=b^*a$, with $a,b\in\{z_i\}$.

\item $S^{N-1}_{\mathbb C,\circ}$: obtained as an intersection, $S^{N-1}_{\mathbb C,\circ}=S^{N-1}_{\mathbb C,\#}\cap S^{N-1}_{\mathbb C,**}$.
\end{enumerate}
\end{definition}

Once again, we use here the general $C^*$-algebra philosophy, which allows us to define noncommutative compact subspaces $S^{N-1}_\times\subset S^{N-1}_{\mathbb C,+}$, by dividing the algebra $C(S^{N-1}_{\mathbb C,+})$ by various algebraic relations, and then by taking the abstract spectrum. See \cite{ba2}.

In addition to the above 4 noncommutative spheres, and to the sphere $S^{N-1}_\mathbb C$ itself, we have as well the following ``subsphere'' of $S^{N-1}_\mathbb C$, which is of interest for us: 
$$\mathbb TS^{N-1}_\mathbb R=\left\{(ux_1,\ldots,ux_N)\in S^{N-1}_\mathbb C\Big|u\in\mathbb T,(x_1,\ldots,x_N)\in S^{N-1}_\mathbb R\right\}$$

Here, and in what follows, $\mathbb T$ is the unit circle in the complex plane.

When adding the above new ``sphere'' to the 5 examples that we have so far, we obtain a set of objects which is stable by intersections, as follows:

\begin{proposition}
We have the following diagram, with all maps being inclusions:
$$\xymatrix@R=15mm@C=15mm{
S^{N-1}_\mathbb C\ar[r]&S^{N-1}_{\mathbb C,**}\ar[r]&S^{N-1}_{\mathbb C,*}\\
\mathbb TS^{N-1}_\mathbb R\ar[r]\ar[u]&S^{N-1}_{\mathbb C,\circ}\ar[u]\ar[r]&S^{N-1}_{\mathbb C,\#}\ar[u]}$$
In addition, this is an intersection diagram, in the sense that any intersection $X\cap Y$ appears on the diagram, as the biggest object contained in both $X,Y$.
\end{proposition}

\begin{proof}
The upper horizontal inclusions are all clear. The lower horizontal inclusions are clear as well, with $\mathbb TS^{N-1}_\mathbb R\subset S^{N-1}_{\mathbb C,\circ}$ coming from the fact that the standard coordinates $z_i=ux_i$ on $\mathbb TS^{N-1}_\mathbb R$ satisfy the relations $ab^*=ba^*=a^*b=b^*a$. 

The two vertical inclusions on the left are clear. The remaining vertical inclusion, on the right, comes from the fact that, by using $ab^*=ba^*,a^*b=b^*a$, we obtain:
$$ab^*c=ba^*c=bc^*a=cb^*a$$

The intersection claim on the right is clear from the definition of $S^{N-1}_{\mathbb C,\circ}$. Regarding the intersection claim on the left, this states that $S^{N-1}_{\mathbb C,-}=S^{N-1}_\mathbb C\cap S^{N-1}_{\mathbb C,\circ}$ equals $\mathbb TS^{N-1}_\mathbb R$.

We have $S^{N-1}_{\mathbb C,-}\subset S^{N-1}_\mathbb C$, so consider a point $z\in S^{N-1}_{\mathbb C,-}$. Since we have $S^{N-1}_{\mathbb C,-}\subset S^{N-1}_{\mathbb C,\#}$, the coordinates of $z$ must satisfy the relations $ab^*=ba^*,a^*b=b^*a$, so we have $z_i\bar{z}_j=z_j\bar{z}_i$. In the case $z_i,z_j\neq0$ we obtain $z_i/\bar{z}_i=z_j/\bar{z}_j$, and we deduce that the numbers $z_i/\bar{z}_i$ are all equal, independently of the index $i$ satisfying $z_i\neq0$. Now by multiplying by a suitable scalar $u\in\mathbb T$, we can assume that we have $z_i/\bar{z}_i=1$, for any $i$ such that $z_i\neq0$. Thus, up to the multiplication by a scalar $u\in\mathbb T$, we have $z\in S^{N-1}_\mathbb R$, as desired.
\end{proof}

As already mentioned, the above 6 spheres were introduced in \cite{ba2}. In order to explain where these spheres come from, let us recall from \cite{ba2} that we have:

\begin{proposition}
We have the following intersection diagram,
$$\xymatrix@R=5mm@C=15mm{
S^{N-1}_{\mathbb R,+}\ar[rrr]&&&S^{N-1}_{\mathbb C,+}\\
&\ar@/^/[dr]&&\\
S^{N-1}_{\mathbb R,*}\ar[uu]\ar@/^/@{-}[ur]&S^{N-1}_\mathbb C\ar[r]&S^{N-1}_{\mathbb C,**}\ar[r]&S^{N-1}_{\mathbb C,*}\ar[uu]\\
\\
S^{N-1}_\mathbb R\ar[r]\ar[uu]&\mathbb TS^{N-1}_\mathbb R\ar[r]\ar[uu]&S^{N-1}_{\mathbb C,\circ}\ar[uu]\ar[r]&S^{N-1}_{\mathbb C,\#}\ar[uu]}$$
with the spheres on the left being the real versions ($z_i=z_i^*$) of the spheres on the right.
\end{proposition}

\begin{proof}
Observe first that $S^{N-1}_\mathbb R$ is the real version of $S^{N-1}_{\mathbb C,\#}$, because when assuming that the coordinates are self-adjoint, the relations $ab^*=ba^*,a^*b=b^*a$ read $ab=ba$. 

Also, we have an inclusion $S^{N-1}_{\mathbb R,*}\subset S^{N-1}_{\mathbb C,**}$, because when taking the real version $S^{N-1}_{\mathbb R,*}$ of the sphere $S^{N-1}_{\mathbb C,*}$, the defining relations $ab^*c=cb^*a$ read $abc=cba$. 

With these observations in hand, the fact that we have the diagram in the statement, and that this is an intersection diagram, are clear from Proposition 1.3.
\end{proof}

The point now is that the above 10 spheres have a number of common features: 

\begin{proposition}
The above $10$ spheres appear from $S^{N-1}_{\mathbb C,+}$ via relations of type
$$z_{i_1}^{e_1}\ldots z_{i_k}^{e_k}=z_{i_\sigma(1)}^{d_1}\ldots z_{i_{\sigma(k)}}^{d_k},\forall i_1,\ldots,i_k$$
where $\sigma\in S_k$ is a permutation, and where $e_i,d_i\in\{1,*\}$ are exponents.
\end{proposition}

\begin{proof}
The 10 spheres appear indeed from $S^{N-1}_{\mathbb C,+}$ via the following relations: 
$$a=a^*,ab=ba,ab^*=b^*a,ab^*=ba^*,a^*b=b^*a$$
$$abc=cba,\ abc^*=c^*ba,\ ab^*c=cb^*a$$

Now since all these relations are as in the statement, this proves the result.
\end{proof}

As explained in \cite{ba2}, the formalism in Proposition 1.5 is in fact too wide. The solution proposed in \cite{ba2} is that of starting with $S^{N-1}_\mathbb R\subset S^{N-1}_{\mathbb R,*}\subset S^{N-1}_{\mathbb R,+}$, which are conjecturally the only real examples, and then by performing 3 operations:
\begin{enumerate}
\item Mirroring: this produces the spheres $S^{N-1}_\mathbb C\subset S^{N-1}_{\mathbb C,**}\subset S^{N-1}_{\mathbb C,+}$.

\item Free complexification: this produces the extra spheres $S^{N-1}_{\mathbb C,\#}\subset S^{N-1}_{\mathbb C,*}$.

\item Taking intersections: this produces the remaining spheres $\mathbb TS^{N-1}_\mathbb R\subset S^{N-1}_{\mathbb C,\circ}$.
\end{enumerate}

Summarizing, the above 10 spheres are expected to be the ``only ones'', under some strong axioms, which are however not available yet. See \cite{ba2}.

Let us try now to better understand the half-liberated spheres. Given $S^{N-1}_\times\subset S^{N-1}_{\mathbb C,+}$, the associated projective space is the quotient $S^{N-1}_\times\to P^N_\times$ given by the fact that $C(P^N_\times)\subset C(S^{N-1}_\times)$ is the subalgebra generated by the variables $p_{ij}=z_iz_j^*$. We have then:

\begin{theorem}
The projective spaces for the $6$ half-liberated spheres are
$$\xymatrix@R=15mm@C=17mm{
P^N_\mathbb C\ar@{=}[r]&P^N_\mathbb C\ar@{=}[r]&P^N_\mathbb C\\
P^N_\mathbb R\ar@{=}[r]\ar[u]&P^N_\mathbb R\ar[u]\ar@{=}[r]&P^N_\mathbb R\ar[u]}$$
where $P^N_\mathbb R,P^N_\mathbb C$ are the usual real and complex projective spaces.
\end{theorem}

\begin{proof}
We use the following presentation results, coming from the Gelfand theorem:
\begin{eqnarray*}
C(P^N_\mathbb R)&=&C^*_{comm}\left((p_{ij})_{i,j=1,\ldots,N}\Big|p=p^t=p^*=p^2,Tr(p)=1\right)\\
C(P^N_\mathbb C)&=&C^*_{comm}\left((p_{ij})_{i,j=1,\ldots,N}\Big|p=p^*=p^2,Tr(p)=1\right)
\end{eqnarray*}

By functoriality, the projective spaces for our 6 spheres are as follows:
$$\xymatrix@R=15mm@C=17mm{
P^N_\mathbb C\ar[r]&P^N_{\mathbb C,**}\ar[r]&P^N_{\mathbb C,*}\\
P^N_\mathbb R\ar[r]\ar[u]&P^N_{\mathbb C,\circ}\ar[u]\ar[r]&P^N_{\mathbb C,\#}\ar[u]}$$

In order to finish, it is enough to prove that we have $P^N_{\mathbb C,*}\subset P^N_\mathbb C$, $P^N_{\mathbb C,\#}\subset P^N_\mathbb R$.

\underline{$P^N_{\mathbb C,*}\subset P^N_\mathbb C$}. From $ab^*c=cb^*a$ we obtain $ab^*cd^*=cb^*ad^*=cd^*ab^*$, so the variables $p_{ij}=z_iz_j^*$ commute. In addition we have $p=p^*=p^2,Tr(p)=1$, and we are done.

\underline{$P^N_{\mathbb C,\#}\subset P^N_\mathbb R$}. From $ab^*=ba^*$ we deduce that the matrix $p_{ij}=z_iz_j^*$ is symmetric, and so $P^N_{\mathbb C,\#}\subset P^N_{\mathbb C,*}=P^N_\mathbb C$ follows to be a subspace of $P^N_\mathbb R$, as desired.
\end{proof}

We should mention that the above result has an extension to the 10-sphere framework of Proposition 1.4, with the 3 rows of spheres corresponding to the 3 types of projective spaces (real, complex, free). Indeed, we have $P^N_{\mathbb R,*}=P^N_\mathbb C$, and the inclusion $P^N_{\mathbb R,+}\subset P^N_{\mathbb C,+}$ is known to be an isomorphism at the level of reduced versions. See \cite{ba2}, \cite{bgo}.

\section{Matrix models}

We further advance now on the understanding of the 6 half-liberated spheres.

Given a subspace $X\subset S^{N-1}_{\mathbb C,+}$, we can consider the subalgebra $C(\widetilde{X})\subset C(\mathbb T)*C(X)$ generated by the elements $w_i=uz_i$, where $u\in C(\mathbb T)$ is the standard generator.  Since we have $\sum_iw_iw_i^*=\sum_iw_i^*w_i=1$, we obtain in this way a closed subspace $\widetilde{X}\subset S^{N-1}_{\mathbb C,+}$, called free complexification of $X$. See \cite{ba1}, \cite{rau}. With this notion in hand, we have:

\begin{proposition}
We have inclusions and equalities as follows,
$$\xymatrix@R=15mm@C=15mm{
\widetilde{S}^{N-1}_\mathbb C\ar[r]&\widetilde{S}^{N-1}_{\mathbb C,*}\ar@{=}[r]&S^{N-1}_{\mathbb C,*}\\
\widetilde{S}^{N-1}_\mathbb R\ar[r]\ar[u]&\widetilde{S}^{N-1}_{\mathbb C,\#}\ar[u]\ar@{=}[r]&S^{N-1}_{\mathbb C,\#}\ar[u]}$$
making correspond standard coordinates to standard coordinates.
\end{proposition}

\begin{proof}
Consider the diagram in Proposition 1.3, with $\mathbb TS^{N-1}_\mathbb R$ replaced by $S^{N-1}_\mathbb R$. By functoriality, we have inclusions as follows:
$$\xymatrix@R=14mm@C=14mm{
\widetilde{S}^{N-1}_\mathbb C\ar[r]&\widetilde{S}^{N-1}_{\mathbb C,**}\ar[r]&\widetilde{S}^{N-1}_{\mathbb C,*}\\
\widetilde{S}^{N-1}_\mathbb R\ar[r]\ar[u]&\widetilde{S}^{N-1}_{\mathbb C,\circ}\ar[u]\ar[r]&\widetilde{S}^{N-1}_{\mathbb C,\#}\ar[u]}$$

Thus we have the square on the left in the statement. In order to prove now the isomorphisms on the right, consider the space $\widetilde{S}^{N-1}_{\mathbb C,*}$, with coordinates $w_i=uz_i$. We have:
$$w_iw_j^*w_k=uz_iz_j^*z_k=uz_kz_j^*z_i=w_kw_j^*w_i$$

Thus we have $\widetilde{S}^{N-1}_{\mathbb C,*}\subset S^{N-1}_{\mathbb C,*}$. As for the converse inclusion, this follows by using the following composition, with $\varepsilon*id$ on the right, where $\varepsilon:C(\mathbb T)\to\mathbb C$ is the counit:
$$C(\widetilde{S}^{N-1}_{\mathbb C,*})\subset C(\mathbb T)*C(S^{N-1}_{\mathbb C,*})\to C(S^{N-1}_{\mathbb C,*})$$

In order to establish now the lower right isomorphism, consider the space $\widetilde{S}^{N-1}_{\mathbb C,\#}$, with coordinates $w_i=uz_i$. We have then $\widetilde{S}^{N-1}_{\mathbb C,\#}\subset S^{N-1}_{\mathbb C,\#}$, because:
$$w_iw_j^*=ux_i\cdot x_j^*u^*=ux_j\cdot x_i^*u^*=w_jw_i^*$$
$$w_i^*w_j=x_i^*u^*\cdot ux_j=x_j^*u^*\cdot ux_i=w_j^*w_i$$

As for the converse inclusion, this follows by using the counit, as before.
\end{proof}

Regarding now $S^{N-1}_{\mathbb C,**},S^{N-1}_{\mathbb C,\circ}$, we can use here some $2\times2$ matrix tricks, inspired from \cite{bdu}. Given a closed subspace $X\subset S^{N-1}_{\mathbb C,+}$, with coordinates denoted $z_i$, we can consider the subalgebra $C(|X|)\subset M_2(C(X))$ generated by the following elements:
$$z_i'=\begin{pmatrix}0&z_i\\ z_i^*&0\end{pmatrix}$$

Since these elements are self-adjoint, and their squares sum up to 1, we have $|X|\subset S^{N-1}_{\mathbb R,+}$. We call this space $|X|$ doubling of $X$. We have then the following result:

\begin{proposition}
We have inclusions and equalities as follows, 
$$\xymatrix@R=15mm@C=15mm{
S^{N-1}_{\mathbb R,*}\ar@{=}[r]&S^{N-1}_{\mathbb R,*}\ar[r]&S^{N-1}_{\mathbb R,+}\\
|S^{N-1}_\mathbb C|\ar[r]\ar[u]&|S^{N-1}_{\mathbb C,*}|\ar[u]\ar[r]&|S^{N-1}_{\mathbb C,+}|\ar[u]}$$
mapping the standard coordinates to the standard coordinates.
\end{proposition}

\begin{proof}
The inclusion on the right appears as the particular case $X=S^{N-1}_{\mathbb C,+}$ of the inclusion $|X|\subset S^{N-1}_{\mathbb R,+}$ constructed above. Regarding now the middle inclusion, we have:
$$z_i'z_j'z_k'=\begin{pmatrix}0&z_i\\ z_i^*&0\end{pmatrix}\begin{pmatrix}0&z_j\\ z_j^*&0\end{pmatrix}\begin{pmatrix}0&z_k\\ z_k^*&0\end{pmatrix}=\begin{pmatrix}0&z_iz_j^*z_k\\ z_i^*z_jz_k^*&0\end{pmatrix}$$

Now by assuming that the elements $z_i$ are the standard coordinates of $S^{N-1}_{\mathbb C,*}$, we conclude that we have $z_i'z_j'z_k'=z_k'z_j'z_i'$, and this gives the middle inclusion. Finally, the inclusion on the left follows by restricting the inclusion in the middle.
\end{proof}

In order to extend the above notions to the complex case, we begin with a technical result, regarding the relation between the real and the complex spheres. 

We denote by $x_i$ the coordinates on the real spheres. In the odd-dimensional case, we can split half-half the coordinates, and denote them $x_i,y_i$. We have then:

\begin{proposition}
We have the following diagram, given by $z_i=x_i+iy_i$,
$$\xymatrix@R=10mm@C=15mm{
S^{2N-1}_\mathbb R\ar[r]&S^{2N-1}_{\mathbb R,*}\ar[r]&S^{2N-1}_{\mathbb R,+}\\
\dot{S}^{2N-1}_\mathbb R\ar[r]\ar@{=}[u]&\dot{S}^{2N-1}_{\mathbb R,*}\ar[r]\ar[u]&\dot{S}^{2N-1}_{\mathbb R,+}\ar[u]\\
S^{N-1}_\mathbb C\ar[r]\ar@{=}[u]&S^{N-1}_{\mathbb C,**}\ar[r]\ar@{=}[u]&S^{N-1}_{\mathbb C,+}\ar@{=}[u]
}$$
where each $\dot{S}^{2N-1}_{\mathbb R,\times}\subset S^{2N-1}_{\mathbb R,\times}$ is obtained via the relations $\sum_i[x_i,y_i]=0$.
\end{proposition}

\begin{proof}
The composition on the left corresponds to the isomorphism $S^{N-1}_\mathbb C=S^{2N-1}_\mathbb R$ given by $z_i=x_i+iy_i$. Observe that we have indeed $\dot{S}^{2N-1}_\mathbb R=S^{2N-1}_\mathbb R$, by commutativity.

We construct now the maps on the right. With $z=x+iy$ we have:
\begin{eqnarray*}
zz^*&=&(x+iy)(x-iy)=x^2+y^2-i[x,y]\\
z^*z&=&(x-iy)(x+iy)=x^2+y^2+i[x,y]
\end{eqnarray*}

Thus, with $z_i=x_i+iy_i$, we have the following formulae:
\begin{eqnarray*}
\sum_iz_iz_i^*&=&\sum_i(x_i^2+y_i^2)-i\sum_i[x_i,y_i]\\
\sum_iz_i^*z_i&=&\sum_i(x_i^2+y_i^2)+i\sum_i[x_i,y_i]
\end{eqnarray*}

We conclude that we have the following equivalence:
$$\sum_iz_iz_i^*=\sum_iz_i^*z_i=1\iff \sum_ix_i^2+y_i^2=1,\sum_i[x_i,y_i]=0$$

But this gives a quotient map $C(S^{2N-1}_{\mathbb R,+})\to C(S^{N-1}_{\mathbb C,+})$, given by $x_i=Re(z_i),y_i=Im(z_i)$, and this map factorizes as $C(S^{2N-1}_{\mathbb R,+})\to C(\dot{S}^{2N-1}_{\mathbb R,+})=C(S^{N-1}_{\mathbb C,+})$, as desired.

Regarding now the middle maps, we must show that, with $z_i=x_i+iy_i$, we have:
$$\Big\{x_i,y_i\ {\rm half-commute}\Big\}\iff\Big\{z_i,z_i^*\ {\rm half-commute}\Big\}$$

The ``$\implies$'' assertion being clear, let us discuss now the ``$\Longleftarrow$'' assertion. Here the half-commutation relations $abc=cba$ with $a,b,c\in\{z_i,z_i^*\}$ can be written as follows, in terms of $a=x+iy,b=z+it,c=u+iv$, with $x,y,z,t,u,v$ self-adjoint:
$$(x+\alpha y)(z+\beta t)(u+\gamma v)=(u+\gamma v)(z+\beta t)(x+\alpha y)\quad\forall\alpha,\beta,\gamma\in\{i,-i\}$$

Now by looking at the real and imaginary parts, we obtain the following system of equations, once again valid for any choice of $\alpha,\beta,\gamma\in\{i,-i\}$:
$$\begin{cases}
(xzu-uzx)+\alpha\beta(ytu-uty)+\beta\gamma(xtv-vtx)+\alpha\gamma(yzv-vzy)=0\\
\alpha(yzu-uzy)+\beta(xtu-utx)+\gamma(xzv-vzx)+\alpha\beta\gamma(ytv-vty)=0
\end{cases}$$

From the 8 possible choices of $\alpha,\beta,\gamma\in\{i,-i\}$, we select now the 4 ones having at most one $-i$ among $\alpha,\beta,\gamma$. The corresponding $4\times4$ determinants being both nonzero, we conclude that the global system, formed by the above $2\times 8=16$ equations, is equivalent to the vanishing of all 8 quantities of type $xzu-uzx$, and we are done.
\end{proof}

Let us go back now to the question of finding a complex analogue of Proposition 2.2. Given a closed subspace $X\subset S^{2N-1}_{\mathbb C,+}$, with coordinates denoted $x_i,y_i$, we can consider the subalgebra $C([X])\subset M_2(C(X))$ generated by the following elements:
$$z_i=\begin{pmatrix}0&x_i\\ x_i^*&0\end{pmatrix}+i\begin{pmatrix}0&y_i\\ y_i^*&0\end{pmatrix}$$

We call this space $[X]$ complex doubling of $X$. Observe that we do not have in general $[X]\subset S^{N-1}_{\mathbb C,+}$, because the formulae $\sum_iz_iz_i^*=\sum_iz_i^*z_i=1$ are not satisfied.

In relation now with $S^{N-1}_{\mathbb C,**},S^{N-1}_{\mathbb C,\circ}$, let us introduce the following manifolds:
\begin{eqnarray*}
\dot{S}^{2N-1}_\mathbb C&=&\left\{(x,y)\in S^{2N-1}_\mathbb C\Big|\sum_ix_i\bar{y}_i\in\mathbb R\right\}\\
\ddot{S}^{2N-1}_\mathbb C&=&\left\{(x,y)\in\dot{S}^{2N-1}_\mathbb C\Big|x_i\bar{x}_j+y_i\bar{y}_j\in\mathbb R,x_i\bar{y}_j-y_i\bar{x}_j\in i\mathbb R\right\}
\end{eqnarray*}

Consider as well the manifold $\mathbb T^2S^{N-1}_\mathbb R\subset S^{2N-1}_\mathbb C$ consisting of the points of the form $u(\lambda p,\mu p)$, with $u\in\mathbb T$, $(\lambda,\mu)\in S^1_\mathbb R\simeq\mathbb T$, and $p\in S^{N-1}_\mathbb R$. We have then:

\begin{theorem}
We have inclusions of noncommutative spaces as follows, 
$$\xymatrix@R=15mm@C=15mm{
[\mathbb TS^{2N-1}_\mathbb R]\ar[r]&[\dot{S}^{2N-1}_\mathbb C]\\
[\mathbb T^2S^{N-1}_\mathbb R]\ar[u]\ar[r]&[\ddot{S}^{2N-1}_\mathbb C]\ar[u]}\qquad
\xymatrix@R=8mm@C=15mm{\\ \longrightarrow\\}\qquad
\xymatrix@R=15mm@C=15mm{
S^{N-1}_\mathbb C\ar[r]&S^{N-1}_{\mathbb C,**}\\
\mathbb TS^{N-1}_\mathbb R\ar[u]\ar[r]&S^{N-1}_{\mathbb C,\circ}\ar[u]}$$
mapping the standard coordinates to the standard coordinates.
\end{theorem}

\begin{proof}
We have to prove that the $2\times2$ matrix model construction $z_i=x_i'+iy_i'$, with $w'=(^0_{\bar{w}}{\ }^w_0)$, induces morphisms of algebras as follows:
$$\xymatrix@R=15mm@C=15mm{
C(S^{N-1}_{\mathbb C,**})\ar[r]\ar[d]&C(S^{N-1}_\mathbb C)\ar[d]\\
C(S^{N-1}_{\mathbb C,\circ})\ar[r]&C(\mathbb TS^{N-1}_\mathbb R)}
\quad
\xymatrix@R=8mm@C=15mm{\\ \longrightarrow\\}\quad
\xymatrix@R=15mm@C=15mm{
M_2(C(\dot{S}^{2N-1}_\mathbb C))\ar[r]\ar[d]&M_2(C(\mathbb TS^{2N-1}_\mathbb R))\ar[d]\\
M_2(C(\ddot{S}^{2N-1}_\mathbb R))\ar[r]&M_2(C(\mathbb T^2S^{N-1}_\mathbb R))}$$

We will first construct the morphism $C(S^{N-1}_{\mathbb C,**})\to M_2(C(\dot{S}^{2N-1}_\mathbb C))$, and then we will obtain the remaining 3 morphisms by factorizing this morphism.

{\bf 1.} We first construct the morphism at top left. We recall from Proposition 2.3 above and its proof that with $z_i=x_i+iy_i$, we have the following equivalence:
$$\sum_iz_iz_i^*=\sum_iz_i^*z_i=1\iff \sum_ix_i^2+y_i^2=1,\sum_i[x_i,y_i]=0$$

In our situation now, with $z_i=x_i'+iy_i'$, and $(x,y)\in\dot{S}^{2N-1}_\mathbb C$, we have:
\begin{eqnarray*}
\sum_ix_i'^2+y_i'^2&=&\sum_i\begin{pmatrix}|x_i|^2&0\\0&|x_i|^2\end{pmatrix}+\begin{pmatrix}|y_i|^2&0\\0&|y_i|^2\end{pmatrix}=\begin{pmatrix}1&0\\0&1\end{pmatrix}\\
\sum_i[x_i',y_i']&=&\sum_i\begin{pmatrix}x_i\bar{y}_i&0\\ 0&\bar{x}_iy_i\end{pmatrix}-\begin{pmatrix}y_i\bar{x}_i&0\\ 0&\bar{y}_ix_i\end{pmatrix}=\begin{pmatrix}0&0\\0&0\end{pmatrix}
\end{eqnarray*}

Thus, we have a morphism $C(S^{N-1}_{\mathbb C,+})\to M_2(C(\dot{S}^{2N-1}_\mathbb C))$. Now since the matrices $x_i',y_i'$ half-commute, the variables $z_i=x_i'+iy_i'$ and their adjoints $z_i^*=x_i'-iy_i'$ half-commute as well, and we therefore obtain a factorization $C(S^{N-1}_{\mathbb C,**})\to M_2(C(\dot{S}^{2N-1}_\mathbb C))$.

{\bf 2.} We prove now that, when restricting attention to $\ddot{S}^{2N-1}_\mathbb C\subset\dot{S}^{2N-1}_\mathbb C$, we obtain a model for $S^{N-1}_{\mathbb C,\circ}\subset S^{N-1}_{\mathbb C,**}$. For this purpose, we recall that $S^{N-1}_{\mathbb C,\circ}\subset S^{N-1}_{\mathbb C,**}$ appears via the relations $ab^*=ba^*,a^*b=b^*a$. With $a=x+iy,b=z+it$, these relations are:
$$\begin{cases}
(x+iy)(z-it)=(z+it)(x-iy)\\
(x-iy)(z+it)=(z-it)(x+iy)
\end{cases}$$

These relations read $[x,z]+[y,t]=\pm i(xt+tx-yz-zy)$, so they are equivalent to:
$$\begin{cases}
[x,z]+[y,t]=0\\ 
xt+tx=yz+zy
\end{cases}$$ 

Now in terms of our variables $z_i=x_i'+iy_i'$, we must have:
$$\begin{cases}
[x_i',x_j']+[y_i',y_j']=0\\
x_i'y_j'+y_j'x_i'=y_i'x_j'+x_j'y_i'
\end{cases}$$

In order to apply these equations to our $2\times2$ matrices, we use the following formula:
$$x'y'=\begin{pmatrix}0&x\\ \bar{x}&0\end{pmatrix}\begin{pmatrix}0&y\\ \bar{y}&0\end{pmatrix}=\begin{pmatrix}x\bar{y}&0\\ 0&\bar{x}y\end{pmatrix}$$

We are therefore led to the following equations, for the parameter space for $S^{N-1}_{\mathbb C,\circ}$:
$$\begin{cases}
x_i\bar{x}_j-x_j\bar{x}_i+y_i\bar{y}_j-y_j\bar{y}_i=0\\
x_i\bar{y}_j+y_j\bar{x}_i=y_i\bar{x}_j+x_j\bar{y}_i
\end{cases}$$

These latter equations can be written more conveniently, as follows:
$$\begin{cases}
x_i\bar{x}_j+y_i\bar{y}_j=x_j\bar{x}_i+y_j\bar{y}_i\\
x_i\bar{y}_j-y_i\bar{x}_j=x_j\bar{y}_i-y_j\bar{x}_i
\end{cases}$$

But these are exactly the equations for $\ddot{S}^{2N-1}_\mathbb C\subset\dot{S}^{2N-1}_\mathbb C$, and we are done.

{\bf 3.} We prove now that, when restricting attention to $\mathbb TS^{2N-1}_\mathbb R\subset\dot{S}^{2N-1}_\mathbb C$, we obtain a model for $S^{N-1}_\mathbb C\subset S^{N-1}_{\mathbb C,**}$. In order to obtain such a model, the variables $z_i,z_i^*$ must commute, and so the variables $x_i',y_i'$ must commute. Thus we must have:
$$x_i\bar{x}_j\in\mathbb R,\quad y_i\bar{y}_j\in\mathbb R,\quad x_i\bar{y}_j\in\mathbb R$$

With $\lambda=||x||,\mu=||y||$ the first two conditions read $x\in\lambda\mathbb TS^{N-1}_\mathbb R,y\in\mu\mathbb TS^{N-1}_\mathbb R$, so let us write $x=\lambda up,y=\mu vq$ with $u,v\in\mathbb T$ and $p,q\in S^{N-1}_\mathbb R$. The third condition tells us then that we must have $u\bar{v}\in\mathbb R$, and so $v=\pm u$, and by changing if necessary $q\to-q$, we can assume that we have $u=v$. We conclude that we have $(x,y)=u(\lambda p,\mu q)$, and since the point $(\lambda p,\mu q)$ must belong to the real sphere $S^{2N-1}_\mathbb R$, we are done.

{\bf 4.} We prove now that $\mathbb T^2S^{N-1}_\mathbb R$ is the model space for $\mathbb TS^{N-1}_\mathbb R$. By functoriality, this latter model space appears as an intersection, $\mathbb TS^{2N-1}_\mathbb R\cap\ddot{S}^{2N-1}_\mathbb C$. So, let us pick a point $(x,y)\in\mathbb TS^{2N-1}_\mathbb R$, and apply to it the equations for $\ddot{S}^{2N-1}_\mathbb C$. These equations are:
$$\begin{cases}
x_i\bar{x}_j+y_i\bar{y}_j\in\mathbb R\\
x_i\bar{y}_j-y_i\bar{x}_j\in i\mathbb R
\end{cases}$$

The first equations are automatic, and since the variables in the second equations are real as well, these equations tell us that we must have $x_i\bar{y}_j=y_i\bar{x}_j$, for any $i,j$. Now with $(x,y)=u(p,q)$ these latter equations read $p_iq_j=q_ip_j$, for any $i,j$. We deduce that we must have $(p,q)=(\lambda r,\mu r)$ with $(\lambda,\mu)\in S^1_\mathbb R$ and $r\in S^{N-1}_\mathbb R$, and we are done.
\end{proof}

As an application of the above methods, we have the following result:

\begin{proposition}
The inclusions between the $6$ half-liberated spheres
$$\xymatrix@R=15mm@C=15mm{
S^{N-1}_\mathbb C\ar[r]&S^{N-1}_{\mathbb C,**}\ar[r]&S^{N-1}_{\mathbb C,*}\\
\mathbb TS^{N-1}_\mathbb R\ar[r]\ar[u]&S^{N-1}_{\mathbb C,\circ}\ar[u]\ar[r]&S^{N-1}_{\mathbb C,\#}\ar[u]}$$
are all proper, at any $N\geq2$.
\end{proposition}

\begin{proof}
By using Theorem 1.6, the vertical maps are all proper. For the horizontal maps, we can use Proposition 2.1, Proposition 2.2 and Theorem 2.4:

\underline{$S^{N-1}_\mathbb C\subset S^{N-1}_{\mathbb C,**}$.} This follows from Proposition 2.2, because the inclusion $|S^{N-1}_\mathbb C|\subset S^{N-1}_{\mathbb R,*}$ found there shows that $S^{N-1}_{\mathbb R,*}$ is not classical. Thus, $S^{N-1}_{\mathbb C,**}$ is not classical either.

\underline{$S^{N-1}_{\mathbb C,\circ}\subset S^{N-1}_{\mathbb C,\#}$.} Here we can use the inclusion $\widetilde{S}^{N-1}_\mathbb R\subset S^{N-1}_{\mathbb C,\#}$ from Proposition 2.1. Indeed, since the standard coordinates $w_i=ux_i$ on the free complexification $\widetilde{S}^{N-1}_\mathbb R$ don't satisfy the relations $abc=cba$, we have $\widetilde{S}^{N-1}_\mathbb R\not\subset S^{N-1}_{\mathbb C,**}$, and so $S^{N-1}_{\mathbb C,\#}\not\subset S^{N-1}_{\mathbb C,**}$, as subspaces of $S^{N-1}_{\mathbb C,+}$. It follows that $(S^{N-1}_{\mathbb C,**}\cap S^{N-1}_{\mathbb C,\#})\subset S^{N-1}_{\mathbb C,\#}$ is indeed proper.

\underline{$S^{N-1}_{\mathbb C,**}\subset S^{N-1}_{\mathbb C,*}$.} Assuming that this inclusion is an equality, by intersecting with $S^{N-1}_{\mathbb C,\#}$ we would obtain that $S^{N-1}_{\mathbb C,\circ}\subset S^{N-1}_{\mathbb C,\#}$ is an equality too, contradiction.

\underline{$\mathbb TS^{N-1}_\mathbb R\subset S^{N-1}_{\mathbb C,\circ}$.} Here we must show that $S^{N-1}_{\mathbb C,\circ}=S^{N-1}_{\mathbb C,**}\cap S^{N-1}_{\mathbb C,\#}$ is not classical. Since we have embeddings between spheres $S^1_{\mathbb C,\times}\subset S^{N-1}_{\mathbb C,\times}$ given by $x_3=x_4=\ldots=x_N=0$, it is enough to solve the problem at $N=2$. So, consider the manifold $\ddot{S}^3_\mathbb C\subset S^3_\mathbb C$ used in Theorem 2.4. The equations defining it, over $(x_1,x_2,y_1,y_2)\in S^3_\mathbb C$, are as follows:
$$\begin{cases}
x_1\bar{y}_1+x_2\bar{y}_2\in\mathbb R\\
x_1\bar{x}_2+y_1\bar{y}_2\in\mathbb R\\
x_1\bar{y}_2-y_1\bar{x}_2\in i\mathbb R
\end{cases}$$

Observe now that these equations are satisfied for the following point:
$$(x_1,x_2,y_1,y_2)=\frac{1}{\sqrt{2}}(i,0,0,1)$$

The corresponding matrices $z_1,z_2$ for this special point are then:
$$z_1=\frac{1}{\sqrt{2}}\begin{pmatrix}0&i\\ -i&0\end{pmatrix}\qquad\qquad 
z_2=\frac{1}{\sqrt{2}}\begin{pmatrix}0&i\\ i&0\end{pmatrix}$$

Now since these two matrices do not commute, this finishes the proof.
\end{proof}

\section{Quantum groups}

In this section and in the next one we further advance on the understanding of the 6 half-liberated spheres, by studying the associated quantum isometry groups.

Our starting point is the following definition, due to Wang \cite{wa1}:

\begin{definition}
The free analogue of $C(U_N)$ is the universal $C^*$-algebra
$$C(U_N^+)=C^*\left((u_{ij})_{i,j=1,\ldots,N}\Big|u,u^t={\rm unitaries}\right)$$
with Hopf algebra maps $\Delta(u_{ij})=\sum_ku_{ik}\otimes u_{kj}$, $\varepsilon(u_{ij})=\delta_{ij}$, $S(u_{ij})=u_{ji}^*$.
\end{definition}

As explained in \cite{wa1}, the above formulae define indeed a comultiplication, counit and antipode, and we have a Hopf $C^*$-algebra in the sense of Woronowicz \cite{wo1}, \cite{wo2}. Observe that the square of the antipode is the identity, $S^2=id$. The underlying noncommutative space $U_N^+$ is a compact quantum group, called free analogue of $U_N$.

Observe the analogy with Definition 1.1. We can build on this analogy, by introducing ``quantum group analogues'' of the spheres in Definition 1.2, simply by imposing the relations there to the standard coordinates of $U_N^+$. We obtain in this way:

\begin{proposition}
We have an intersection diagram of compact quantum groups
$$\xymatrix@R=15mm@C=17mm{
U_N\ar[r]&U_N^{**}\ar[r]&U_N^*\\
\mathbb TO_N\ar[r]\ar[u]&U_N^\circ\ar[u]\ar[r]&U_N^\#\ar[u]}$$
with $U_N^*,U_N^{**},U_N^\#,U_N^\circ$ being defined inside $U_N^+$ via the relations in Definition 1.2.
\end{proposition}

\begin{proof}
The quantum groups $U_N^*,U_N^{**}$ were introduced and studied in \cite{bdd}, \cite{bdu}. Regarding $U_N^\#$, our first claim is that its defining relations can be reformulated as follows:
$$\Big(ab^*=ba^*,a^*b=b^*a\Big)\iff\Big(ab^*c\ {\rm depends\ only\ on\ } \{a,b,c\}\Big)$$

Indeed, the implication ``$\implies$'' can be checked by alternatively using the relations $ab^*=ba^*,a^*b=b^*a$, on left and on the right, as follows:
$$ab^*c=ba^*c=bc^*a=cb^*a=ca^*b=ac^*b$$

As for the converse implication, ``$\Longleftarrow$'', the first formula follows from the following computation, and the proof of the second formula is similar:
$$ab^*c=ba^*c\implies\sum_cab^*cc^*=\sum_cba^*cc^*\implies ab^*=ba^*$$

With the above claim in hand, the construction of $\varepsilon,S$ is clear. Concerning now the comultiplication $\Delta$, observe that with $U_{ij}=\sum_ku_{ik}\otimes u_{kj}$, we have:
$$U_{ix}U_{jy}^*U_{kz}=\sum_{abc}u_{ia}u_{jb}^*u_{kc}\otimes u_{ax}u_{by}^*u_{cz}$$

Now let us permute $(ix),(jy),(kz)$. We can use the same permutation $\sigma\in S_3$ in order to permute $a,b,c$, in a similar way, and this gives the existence of $\Delta$. 

Finally, if we set $U_N^\circ=U_N^{**}\cap U_N^\#$, we obtain as well a compact quantum group.

Thus, we have the 6 quantum groups in the statement. The inclusions are clear, and the intersection claim $U_N\cap U_N^\circ=\mathbb TO_N$ follows as in the proof of Proposition 1.3.
\end{proof}

We have as well analogues of the other basic results regarding spheres. First, we have the following analogue of Theorem 1.6 above, basically known since \cite{bdd}:

\begin{proposition}
The projective versions of the $6$ quantum groups are:
$$\xymatrix@R=14mm@C=14mm{
PU_N\ar@{=}[r]&PU_N\ar@{=}[r]&PU_N\\
PO_N\ar@{=}[r]\ar[u]&PO_N\ar@{=}[r]\ar[u]&PO_N\ar[u]}$$
In addition, we have $PO_N=O_N/\mathbb Z_2$ and $PU_N=U_N/\mathbb T$.
\end{proposition}

\begin{proof}
By functoriality, it is enough to prove that we have inclusions $PU_N^*\subset PU_N$ and $PU_N^\#\subset PO_N$. As explained in \cite{bdd}, the first inclusion can be deduced as follows:
$$PU_N^*\subset(PU_N^*)_{class}\subset(PU_N^+)_{class}=PU_N$$

Indeed, the first inclusion follows from the fact that the projective version coordinates $w_{ia,jb}=u_{ij}u_{ab}^*$ commute, the second inclusion follows by functoriality from $U_N^*\subset U_N^+$, and the third inclusion follows from Tannakian duality, as explained in \cite{bve}.

Regarding now the second inclusion, this follows from $PU_N^\#\subset PU_N^*\subset PU_N$, and from the fact that the variables $w_{ia,jb}=u_{ij}u_{ab}^*$ are self-adjoint over $PU_N^\#$.
\end{proof}

Regarding now the free complexifications, we have the following result, which is much more precise than the one for the spheres, from Proposition 2.1 above:

\begin{proposition}
The free complexifications of the $6$ quantum groups are
$$\xymatrix@R=14mm@C=14mm{
U_N^*\ar@{=}[r]&U_N^*\ar@{=}[r]&U_N^*\\
U_N^\#\ar@{=}[r]\ar[u]&U_N^\#\ar@{=}[r]\ar[u]&U_N^\#\ar[u]}$$
with all the isomorphisms mapping standard coordinates to standard coordinates.
\end{proposition}

\begin{proof}
The arguments in the proof of Proposition 2.1 extend to the quantum group case, and provide us with the following diagram:
$$\xymatrix@R=15mm@C=15mm{
\widetilde{U}_N\ar[r]&\widetilde{U}_N^*\ar@{=}[r]&U_N^*\\
\widetilde{O}_N\ar[r]\ar[u]&\widetilde{U}_N^\#\ar[u]\ar@{=}[r]&U_N^\#\ar[u]}$$

We must prove now that we have $\widetilde{O}_N=U_N^\#$, $\widetilde{U}_N=U_N^*$. For this purpose, we can use Proposition 3.3, and technology from \cite{ba1}. Indeed, since the projective version $PU_N^*$ is classical, we obtain that $U_N^*$, as well as all its subgroups, are amenable. Thus, we can indeed use the results in \cite{ba1}, established there at the level of reduced versions.

With notations and terminology from \cite{ba1}, the quantum groups $\widetilde{U}_N,U_N^*,U_N^\#$ are all easy (called ``free'' there), of infinite level, and appear as free complexifications. Thus the main result in \cite{ba1} applies, and shows that these 3 quantum groups must appear as free complexifications of certain intermediate easy quantum groups $O_N\subset O_N^\times\subset O_N^+$. 

On the other hand, we know from \cite{bve} that the only non-trivial intermediate easy quantum group $O_N\subset G\subset O_N^+$ is the half-liberation $G=O_N^*$. Thus, each of the 3 quantum groups $O_N^\times$ constructed above must satisfy $O_N^\times\in\{O_N,O_N^*,O_N^+\}$. 

In order to finish we use the fact, once again from \cite{bve}, that the projective versions of the quantum groups $O_N\subset O_N^*\subset O_N^+$ are the quantum groups $PO_N\subset PU_N\subset PO_N^+$. In particular, the projective version determines the quantum group. Now since we have $P\widetilde{U}_N=PU_N$, $PU_N^*=PU_N$, $PU_N^\#=PO_N$, we conclude that we have:
$$\widetilde{U}_N=\widetilde{O}_N^*,\qquad U_N^*=\widetilde{O}_N^*,\qquad U_N^\#=\widetilde{O}_N$$

Thus we have indeed $\widetilde{O}_N=U_N^\#$, $\widetilde{U}_N=U_N^*$, and we are done.
\end{proof}

Let us discuss now the analogues of the matrix model constructions from section 2 above. Following \cite{bdu}, we consider the following compact group:
$$U_{2,N}=\left\{\begin{pmatrix}A&B\\-B&A\end{pmatrix}\in U_{2N}\Big|A,B\in M_N(\mathbb C)\right\}$$

We have then the following result, basically from \cite{bdu}:

\begin{proposition}
We have a morphism $C(U_N^{**})\to M_2(C(U_{2,N}))$, given by
$$u_{ij}\to\begin{pmatrix}0&a_{ij}\\ \bar{a}_{ij}&0\end{pmatrix}+i\begin{pmatrix}0&b_{ij}\\ \bar{b}_{ij}&0\end{pmatrix}$$
where $a_{ij},b_{ij}$ denote the standard coordinates on $U_{2,N}$.
\end{proposition}

\begin{proof}
The group elements $U\in U_{2,N}$, written $U=(^{\ A}_{-B}{\ }^B_A)$ as above, satisfy the relations $UU^*=U^*U=U^t\bar{U}=\bar{U}U^t=1$, and we deduce that the matrices $A,B$ satisfy:
$$AA^*+BB^*=A^*A+B^*B=A^t\bar{A}+B^t\bar{B}=\bar{A}A^t+\bar{B}B^t=1$$
$$AB^*=BA^*,\ A^*B=B^*A,\ A^t\bar{B}=B^t\bar{A},\ \bar{A}B^t=\bar{B}A^t$$

Consider now the target elements $w_{ij}=a_{ij}'+ib_{ij}'$ appearing in the statement. The matrix $w=(w_{ij})$ that they form, and its adjoint, are then given by:
$$w=\begin{pmatrix}0&A+iB\\ \bar{A}+i\bar{B}&0\end{pmatrix}\qquad 
w^*=\begin{pmatrix}0&A^t-iB^t\\ A^*-iB^*&0\end{pmatrix}$$

Also, the transpose of this matrix, and its complex conjugate, are given by:
$$w^t=\begin{pmatrix}0&A^*+iB^*\\ A^t+iB^t&0\end{pmatrix}\qquad 
\bar{w}=\begin{pmatrix}0&\bar{A}-i\bar{B}\\ A-iB&0\end{pmatrix}$$

By using now the above formulae relating $A,B$, we obtain:
$$ww^*=w^*w=w^t\bar{w}=\bar{w}w^t=1$$

Thus, we have obtained a morphism of algebras $C(U_N^+)\to M_2(C(U_{2,N}))$. 

Now since the $2\times2$ matrices $a_{ij}',b_{ij}'$ half-commute, so do the elements $w_{ij},w_{ij}^*$, and so our morphism factorizes through the algebra $C(U_N^{**})$, as claimed.
\end{proof}

With the above result in hand, we can suitably modify the ``complex doubling'' operation $X\to[X]$ constructed in section 2 above, as follows:

\begin{definition}
Given $X\subset U_{2,N}$, we define $[[X]]\subset U_N^{**}$ by stating that $C([[X]])$ is the image of the representation $C(U_N^{**})\to M_2(C(X))$, given by $u_{ij}\to a_{ij}'+ib_{ij}'$.
\end{definition} 

In other words, our construction is defined by the following diagram:
$$\xymatrix@R=15mm@C=15mm{
C(U_N^{**})\ar@.[d]\ar[r]&M_2(C(U_{2,N}))\ar[d]\\
C([[X]])\ar@.[r]&M_2(C(X))
}$$

As an example here, the results in \cite{bdu} show that we have $[[U_{2,N}]]=U_N^{**}$.

We can now formulate an analogue of Theorem 2.4 above, as follows:

\begin{theorem}
We have inclusions of noncommutative spaces 
$$\xymatrix@R=15mm@C=15mm{
[[\mathbb TO_{2,N}]]\ar[r]&[[U_{2,N}]]\\
[[\mathbb T^2O_N]]\ar[u]\ar[r]&[[U_N']]\ar[u]}\qquad
\xymatrix@R=8mm@C=15mm{\\ \longrightarrow\\}\qquad
\xymatrix@R=16mm@C=16mm{
U_N\ar[r]&U_N^{**}\\
\mathbb TO_N\ar[u]\ar[r]&U_N^\circ\ar[u]}$$
with $\mathbb T^2O_N$ and $U_N'$ being certain closed subgroups of $U_{2,N}$.
\end{theorem}

\begin{proof}
We follow the method in the proof of Theorem 2.4. The computations there apply to the present situation, with a $2N$ rescaling factor for the spheres, and we obtain that the ``parameter spaces'' for the quantum groups $G=U_N,U_N^\circ,\mathbb TO_N$, i.e. the biggest closed subspaces $X\subset U_{2,N}$ producing embeddings $[[X]]\subset G$, are as follows:
\begin{eqnarray*}
U_N&\to&U_{2,N}\cap 2N\cdot\mathbb TS^{4N^2-1}_\mathbb R\\
U_N^\circ&\to&U_{2,N}\cap 2N\cdot\ddot{S}^{4N^2-1}_\mathbb C\\
\mathbb TO_N&\to&U_{2,N}\cap 2N\cdot\mathbb T^2S^{2N^2-1}_\mathbb R
\end{eqnarray*}

We will compute these three spaces, and then show that they are indeed groups.

{\bf 1.} We first compute the parameter space for $U_N$. We know that a matrix $U\in U_{2,N}$ belongs to this space precisely when there exists $z\in\mathbb T$ such that $V=zU$ is real. Thus $V$ must belong to the group $O_{2,N}=U_{2,N}\cap O_{2N}$, and the parameter space is:
$$\mathbb TO_{2,N}=\left\{z\begin{pmatrix}A&B\\-B&A\end{pmatrix}\in U_{2N}\Big|z\in\mathbb T,A,B\in M_N(\mathbb R)\right\}$$

{\bf 2.} Regarding now the parameter space for $U_N^\circ$, this appears from $U_{2,N}$ via the defining relations for $\ddot{S}^{4N^2-1}_\mathbb C$, from section 2 above, which are as follows:
$$\begin{cases}
a_{ij}\bar{a}_{kl}+b_{ij}\bar{b}_{kl}\in\mathbb R\\
a_{ij}\bar{b}_{kl}-b_{ij}\bar{a}_{kl}\in i\mathbb R
\end{cases}$$

{\bf 3.} Finally, the parameter space for $\mathbb TO_N$ is best obtained by intersecting the parameter spaces for $U_N,U_N^\circ$. Indeed, let us pick a matrix $U\in\mathbb TO_{2,N}$, written $U=z(^{\ A}_{-B}{\ }^B_A)$ as above. Then $U$ belongs to the parameter space for $\mathbb TO_N$ when its entries $\widetilde{a}_{ij}=za_{ij},\widetilde{b}_{ij}=zb_{ij}$ satisfy the above two equations. As in the sphere case, the variable $z\in\mathbb T$ cancels, and the first equation is automatic, and the second equation reads $a_{ij}b_{kl}=b_{ij}a_{kl}$. We therefore conclude, as in the sphere case, that the parameter space for $\mathbb TO_N$ is:
$$\mathbb T^2O_N=\left\{z\begin{pmatrix}cA&sA\\-sA&cA\end{pmatrix}\Big|z\in\mathbb T,\begin{pmatrix}c&s\\-s&c\end{pmatrix}\in SO_2\simeq\mathbb T,A\in O_N\right\}$$

{\bf 4.} We are left with checking that the parameter spaces are indeed groups. Since this is clear for $\mathbb TO_{2,N},\mathbb T^2O_N$, it remains to verify that the following space is a group:
$$U_N'=\left\{\begin{pmatrix}A&B\\-B&A\end{pmatrix}\in U_{2N}\Big|\ \begin{matrix}a_{ij}\bar{a}_{kl}+b_{ij}\bar{b}_{kl}\in\mathbb R\\
a_{ij}\bar{b}_{kl}-b_{ij}\bar{a}_{kl}\in i\mathbb R\end{matrix}\right\}$$

We have $1\in U_N'$, and $U\in U_N'\implies U^*\in U_N'$ is clear as well, because at the level of coordinates, the passage $U\to U^*$ is given by $(a_{ij},b_{ij})\to(\bar{a}_{ji},-\bar{b}_{ji})$, and this transformation preserves the solutions of the defining equations for $U_N'$.

Regarding now the multiplication axiom, we use the following formula:
$$\begin{pmatrix}A&B\\-B&A\end{pmatrix}\begin{pmatrix}C&D\\-D&C\end{pmatrix}=\begin{pmatrix}AC-BD&AD+BC\\-AD-BC&AC-BD\end{pmatrix}$$

Assuming now that the two matrices on the left belong to $U_N'$, we have:
\begin{eqnarray*}
&&(AC-BD)_{ij}\overline{(AC-BD)}_{kl}+(AD+BC)_{ij}\overline{(AD+BC)}_{kl}\\
&=&\sum_{pq}(a_{ip}c_{pj}-b_{ip}d_{pj})(\bar{a}_{kq}\bar{c}_{ql}-\bar{b}_{kq}\bar{d}_{ql})+(a_{ip}d_{pj}+b_{ip}c_{pj})(\bar{a}_{kq}\bar{d}_{ql}-\bar{b}_{kq}\bar{c}_{ql})\\
&=&\sum_{pq}(a_{ip}\bar{a}_{kq}+b_{ip}\bar{b}_{kq})(c_{pj}\bar{c}_{ql}+d_{pj}\bar{d}_{ql})+(a_{ip}\bar{b}_{kq}-b_{ip}\bar{a}_{kq})(d_{pj}\bar{c}_{ql}-c_{pj}\bar{d}_{ql})
\end{eqnarray*}

Now since the above 4 quantities are respectively in $\mathbb R,\mathbb R,i\mathbb R,i\mathbb R$, the summand is real, and hence the whole sum is real as well. Thus, we have checked the first equations.

For the second equations, the proof is similar. We have indeed: 
\begin{eqnarray*}
&&(AC-BD)_{ij}\overline{(AD+BC)}_{kl}-(AD+BC)_{ij}\overline{(AC-BD)}_{kl}\\
&=&\sum_{pq}(a_{ip}c_{pj}-b_{ip}d_{pj})(\bar{a}_{kq}\bar{d}_{ql}+\bar{b}_{kq}\bar{c}_{ql})-(a_{ip}d_{pj}+b_{ip}c_{pj})(\bar{a}_{kq}\bar{c}_{ql}-\bar{b}_{kq}\bar{d}_{ql})\\
&=&\sum_{pq}(a_{ip}\bar{a}_{kq}+b_{ip}\bar{b}_{kq})(c_{pj}\bar{d}_{ql}-d_{pj}\bar{c}_{ql})+(a_{ip}\bar{b}_{kq}-b_{ip}\bar{a}_{kq})(c_{pj}\bar{c}_{ql}-d_{pj}\bar{d}_{ql})
\end{eqnarray*}

Now the quantities which appear are respectively in $\mathbb R,i\mathbb R,i\mathbb R,\mathbb R$, so the summand is imaginary, and hence the whole sum is imaginary as well, and we are done.
\end{proof}

\section{Affine isometries}

In this section we show that the 6 quantum groups introduced above appear as affine quantum isometry groups of the 6 spheres, and we deduce some consequences.

We use the following formalism, inspired from \cite{gos}:

\begin{definition}
We say that $G\subset U_N^+$ acts affinely on $X\subset S^{N-1}_{\mathbb C,+}$ when
$$z_i\to\sum_au_{ia}\otimes z_a$$
defines a morphism of algebras $\Phi:C(X)\to C(G)\otimes C(X)$.
\end{definition}

Observe that such a morphism $\Phi$ is automatically coassociative and counital, in the sense that we have $(id\otimes\Phi)\Phi=(\Delta\otimes id)\Phi$ and $(\varepsilon\otimes id)\Phi=id$. Thus, we have a coaction, in the usual sense. The basic example is $U_N\curvearrowright S^{N-1}_\mathbb C$, via $\Phi(f)(U,x)=f(Ux)$. 

We agree to denote the 6 half-liberated quantum groups by $U_N^\times$, and the corresponding 6 half-liberated spheres by $S^{N-1}_\times$. First, we have the following result:

\begin{proposition}
We have an affine action $U_N^\times\curvearrowright S^{N-1}_\times$.
\end{proposition}

\begin{proof}
We must prove that the formula in Definition 4.1 defines a morphism of algebras $C(S^{N-1}_\times)\to C(U_N^\times)\otimes C(S^{N-1}_\times)$. For this purpose, we just have to show that the elements $Z_i=\sum_au_{ia}\otimes z_a$ satisfy the defining relations for $S^{N-1}_\times$. 

As a first observation, the quadratic relations $\sum_iZ_iZ_i^*=\sum_iZ_i^*Z_i=1$ follow from the biunitarity of $u$. For the remaining relations, we perform a case-by case analysis.

\underline{$S^{N-1}_{\mathbb C,**},S^{N-1}_{\mathbb C,*}$}. For $S^{N-1}_{\mathbb C,*}$ we have indeed the following computation:
$$Z_iZ_j^*Z_k=\sum_{abc}u_{ia}u_{jb}^*u_{kc}\otimes z_az_b^*z_c=\sum_{abc}u_{kc}u_{jb}^*u_{ia}\otimes z_cz_b^*z_a=Z_kZ_j^*Z_i$$

For $S^{N-1}_{\mathbb C,**}$ the proof is similar, by removing all the $*$ exponents.

\underline{$S^{N-1}_{\mathbb C,\circ},S^{N-1}_{\mathbb C,\#}$}. It is enough to do the verification for $S^{N-1}_{\mathbb C,\#}$, and here we have:
$$Z_iZ_j^*=\sum_{ab}u_{ia}u_{jb}^*\otimes z_az_b^*=\sum_{ab}u_{jb}u_{ia}^*\otimes z_bz_a^*=Z_jZ_i^*$$

The proof of $Z_i^*Z_j=Z_j^*Z_i$ is similar, by moving the $*$ exponents on the left.

\underline{$\mathbb TS^{N-1}_\mathbb R,S^{N-1}_\mathbb C$}. It is enough to do the verification for $S^{N-1}_\mathbb C$. But the result here is clear, because $U_N$ is known to act on $S^{N-1}_\mathbb C$, with coaction map as in the statement.
\end{proof}

We will prove now that the actions in Proposition 4.2 are universal. For this purpose, we use an old 3-step method from \cite{bhg}, where the result was established for $S^{N-1}_\mathbb C$. The idea is to: (1) establish linear independence results for the products of coordinates, (2) deduce from this the precise conditions on $G\subset U_N^+$ which allow an action, and (3) solve the quantum group question left, by using an antipode/relabel trick.

In our case, the linear independence lemma that we will need is:

\begin{lemma}
The following variables are linearly independent:
\begin{enumerate}
\item $\{z_az_b^*|1\leq a\leq b\leq N\}$, over $S^{N-1}_{\mathbb C,\circ}$.

\item $\{z_az_bz_c|1\leq a\leq c\leq N,1\leq b\leq N\}$, over $S^{N-1}_{\mathbb C,\circ}$.

\item $\{z_az_b^*z_c|1\leq a\leq c\leq N,1\leq b\leq N\}$, over $S^{N-1}_{\mathbb C,**}$.
\end{enumerate}
\end{lemma}

\begin{proof}
This follows by using various $2\times2$ matrix models for the spheres:

(1) Here we can use the isomorphism $P^{N-1}_{\mathbb C,\circ}\simeq P^N_\mathbb R$ given by $p_{ab}=z_az_b^*$. Indeed, since the variables $\{p_{ab}|a\leq b\}$ are linearly independent over $P^N_\mathbb R$, this gives the result.

(2) We use here the model $z=x'+iy'$, with $(x,y)\in\ddot{S}^{2N-1}_\mathbb C$, found in Theorem 2.4 above. Our first claim is that we have an inclusion, as follows:
$$(S^{N-1}_\mathbb R)^2\subset\ddot{S}^{2N-1}_\mathbb C,\qquad (p,q)\to\left(\frac{p+q}{2},\frac{p-q}{2i}\right)$$

Indeed, since for $p,q\in S^{N-1}_\mathbb R$ we have $\sum_i\left(\frac{p_i+q_i}{2}\right)^2+\left(\frac{p_i-q_i}{2}\right)^2=1$, we obtain an embedding $(S^{N-1}_\mathbb R)^2\subset S^{2N-1}_\mathbb C$. Moreover, since $\sum_i\frac{p_i+q_i}{2}\cdot\frac{p_i-q_i}{2}=0$, we have in fact $(S^{N-1}_\mathbb R)^2\subset\dot{S}^{2N-1}_\mathbb C$, and finally the defining relations for $\ddot{S}^{2N-1}_\mathbb C$ are both trivially satisfied.

When restricting the parameter space to $(S^{N-1}_\mathbb R)^2$, the model becomes:
$$z_i=\frac{1}{2}\begin{pmatrix}0&p_i+q_i\\ p_i+q_i&0\end{pmatrix}+\frac{1}{2}\begin{pmatrix}0&p_i-q_i\\ q_i-p_i&0\end{pmatrix}
=\begin{pmatrix}0&p_i\\ q_i&0\end{pmatrix}$$

Observe now that we have the following formula:
$$z_iz_jz_k=\begin{pmatrix}0&p_i\\ q_i&0\end{pmatrix}\begin{pmatrix}0&p_j\\ q_j&0\end{pmatrix}\begin{pmatrix}0&p_k\\ q_k&0\end{pmatrix}=\begin{pmatrix}0&p_iq_jp_k\\ q_ip_jq_k&0\end{pmatrix}$$

Now since the variables $\{p_iq_jp_k|i\leq k\}$ on the right are linearly independent over $(S^{N-1}_\mathbb R)^2$, so are the $2\times 2$ matrices $\{z_iz_jz_k|i\leq k\}$, and this gives the result.

(3) Here we can use the model $z=x'+iy'$, with $(x,y)\in\dot{S}^{2N-1}_\mathbb C$, from Theorem 2.4, with the parameter space restricted to $S^{N-1}_\mathbb C\simeq S^{2N-1}_R\subset\dot{S}^{2N-1}_\mathbb R$. Indeed, if we denote by $w_i=x_i+iy_i$ the coordinates on $S^{N-1}_\mathbb C$, the matrix model formula becomes:
$$z_i=\begin{pmatrix}0&x_i+iy_i\\ x_i+iy_i&0\end{pmatrix}=\begin{pmatrix}0&w_i\\ w_i&0\end{pmatrix}$$

Now observe that we have the following formula:
$$z_iz_j^*z_k=\begin{pmatrix}0&w_i\\ w_i&0\end{pmatrix}\begin{pmatrix}0&\bar{w}_j\\ \bar{w}_j&0\end{pmatrix}\begin{pmatrix}0&w_k\\ w_k&0\end{pmatrix}=\begin{pmatrix}0&w_i\bar{w}_jw_k\\ w_i\bar{w}_jw_k&0\end{pmatrix}$$

Now since the variables $\{w_i\bar{w}_jw_k|i\leq k\}$ on the right are linearly independent over $S^{N-1}_\mathbb C$, so are the $2\times 2$ matrices $\{z_iz_j^*z_k|i\leq k\}$, and this gives the result.
\end{proof}

We will need as well, several times, the following lemma:

\begin{lemma}
If the standard coordinates $u_{ij}$ on a compact quantum group $G\subset U_N^*$ satisfy the relations $abc=cba$, then we have $G\subset U_N^{**}$.
\end{lemma}

\begin{proof}
We must prove that $abc=cba$ for any $a,b,c\in\{u_{ij},u_{ij}^*\}$, and by using the involution, it is enough to check that the following relations hold, for any $a,b,c\in\{u_{ij}\}$:
$$abc=cba,\qquad ab^*c=cb^*a,\qquad abc^*=c^*ba$$ 

The first two relations hold by assumption, and we must therefore deduce the third relations from them. For this purpose, we can use the diagrammatic formalism in \cite{bsp}, or rather its unitary extension, which applies to the easy quantum group $G\subset U_N^+$ coming from the first two relations. Indeed, in the Tannakian category of $G$, we have:
$$\xymatrix@R=10mm@C=5mm{\bullet\ar@{-}[d]\ar@/^/@{-}[r]&\circ\ar@{-}[drr]&\circ\ar@{-}[d]&\circ\ar@{-}[dll]&\bullet\ar@{-}[d]\\\bullet&\circ&\circ&\circ\ar@/_/@{-}[r]&\bullet}\quad\xymatrix@R=4mm@C=6mm{&\\ =\\&\\& }\xymatrix@R=10mm@C=6mm{\circ\ar@{-}[drr]&\circ\ar@{-}[d]&\bullet\ar@{-}[dll]\\\bullet&\circ&\circ}$$
\vskip-5mm

Thus the relations $abc=cba$ imply the relations $abc^*=c^*ba$, and we are done.
\end{proof}

Now back to the quantum isometries, and to the 3-step method from \cite{bhg}, Lemma 4.3 and Lemma 4.4 provide us with the first step. We will perform the second and third step altogether, first for $S^{N-1}_{\mathbb C,**}\subset S^{N-1}_{\mathbb C,*}$, and then for $S^{N-1}_{\mathbb C,\circ}\subset S^{N-1}_{\mathbb C,\#}$. First, we have:

\begin{proposition}
The affine actions of $U_N^{**},U_N^*$ on $S^{N-1}_{\mathbb C,**},S^{N-1}_{\mathbb C,*}$ are universal.
\end{proposition}

\begin{proof}
This is a routine computation, based on the antipode/relabel trick in \cite{bhg}. Consider indeed a compact quantum group $G\subset U_N^+$, and let $Z_i=\sum_au_{ia}\otimes z_a$. With Lemma 4.4 in mind, let us fix as well a symbol $\times\in\{\emptyset,*\}$. We have then:
$$Z_iZ_j^\times Z_k=\sum_{abc}u_{ia}u_{jb}^\times u_{kc}\otimes z_az_b^\times z_c$$

Assuming now that the variables $z_1,\ldots,z_N$ are subject to the relations $z_az_b^\times z_c=z_cz_b^\times z_a$, some of the terms on the right coincide. By taking into account the various cases, and by merging these terms, we can write the above formula as follows:
\begin{eqnarray*}
Z_iZ_j^\times Z_k
&=&\sum_{a<c,b\neq a,c}(u_{ia}u_{jb}^\times u_{kc}+u_{ic}u_{jb}^\times u_{ka})\otimes z_az_b^\times z_c\\
&+&\sum_{a<c}(u_{ia}u_{ja}^\times u_{kc}+u_{ic}u_{ja}^\times u_{ka})\otimes z_az_a^\times z_c\\
&+&\sum_{a\neq b}u_{ia}u_{jb}^\times u_{ka}\otimes z_az_b^\times z_a\\
&+&\sum_au_{ia}u_{ja}^\times u_{ka}\otimes z_az_a^\times z_a 
\end{eqnarray*}

By interchanging $i\leftrightarrow k$, we have as well a similar formula for $Z_kZ_j^\times Z_i$.

Now by using the linear independence of the variables on the right, coming from Lemma 4.3 (2) and (3) above, we conclude that the relations $Z_iZ_j^\times Z_k=Z_kZ_j^\times Z_i$ are equivalent to the following system of equations, where $[x,y,z]=xy^\times z-zy^\times x$:

(1) $[u_{ia},u_{jb},u_{kc}]=[u_{ka},u_{jb},u_{ic}]$, for $a,b,c$ distinct.

(2) $[u_{ia},u_{ja},u_{kc}]=[u_{ka},u_{ja},u_{ic}]$.

(3) $[u_{ia},u_{jb},u_{ka}]=0$.

Here we have merged the relation coming by comparing the fourth sums, namely $[u_{ia},u_{ja},u_{ka}]=0$ for any $a$, with the relations coming from the second and third sums, in order to drop the assumptions $a\neq c$, $a\neq b$ appearing there.

Our claim, which will prove the result, is that the above equations (1,2,3) are in fact equivalent to $[u_{ia},u_{jb},u_{kc}]=0$, regardless of the indices $i,j,k$ and $a,b,c$.

Let us first process the relations (1). By applying the antipode and then the involution we obtain $[u_{ck},u_{bj},u_{ai}]=[u_{ci},u_{bj},u_{ak}]$, for any $a,b,c$ distinct, and then by relabelling we obtain $[u_{kc},u_{jb},u_{ia}]=[u_{ka},u_{jb},u_{ic}]$, for any $i,j,k$ distinct. Now by comparing with the original relations (1), we have several cases, and we are led to the following relations:

(1a) $[u_{ia},u_{jb},u_{kc}]=0$, for $a,b,c$ distinct, and $i,j,k$ distinct.

(1b) $[u_{ia},u_{jb},u_{kc}]=[u_{ka},u_{jb},u_{ic}]$, for $a,b,c$ distinct, and $i,j,k$ not distinct.

We further process now the relations (1b). Since the relations at $i=k$ are trivial, and those at $i=j,j=k$ are equivalent, we can assume that we have $i=j$, and we get:

(1b') $[u_{ia},u_{ib},u_{kc}]=[u_{ka},u_{ib},u_{ic}]$, for $a,b,c$ distinct.

Let us process now the above relations (2). By applying the antipode and the involution we obtain $[u_{ck},u_{aj},u_{ai}]=[u_{ci},u_{aj},u_{ak}]$, and by relabelling, we obtain:

(2') $[u_{kc},u_{ib},u_{ia}]=[u_{ka},u_{ib},u_{ic}]$.

The point now is that the relations (1b'), (2') can be merged. Indeed, in view of (2'), the relations (1b') simplify to:

(1b'') $[u_{ia},u_{ib},u_{kc}]=0$, for $a,b,c$ distinct.

Now, with these relations (1b'') in hand, the relations (2') are automatic for $a,b,c$ distinct. Thus, what is left from the relations (2') is:

(2'') $[u_{kc},u_{ib},u_{ia}]=[u_{ka},u_{ib},u_{ic}]$, for $a,b,c$ not distinct.

As a partial conclusion, the relevant relations are (1a), (1b''), (2''), (3). Now let us further process the relations (2"). Since these relations are automatic at $a=c$, and are equivalent at $a=b,b=c$, we can assume $a=b$, and we obtain:

($2^*$) $[u_{ka},u_{ia},u_{ic}]=[u_{kc},u_{ia},u_{ia}]$.

Now by applying the antipode and then the involution we obtain $[u_{ci},u_{ai},u_{ak}]=[u_{ai},u_{ai},u_{ck}]$, and by relabelling we obtain $[u_{ka},u_{ia},u_{ic}]=[u_{ia},u_{ia},u_{kc}]$. By comparing now with the original relations ($2^*$) we are led to the following two relations:

($2^*$a) $[u_{ka},u_{ia},u_{ic}]=0$.

($2^*$b) $[u_{kc},u_{ia},u_{ia}]=0$.

Summarizing, the relevant relations are (1a), (1b''), ($2^*$a), ($2^*$b), (3). Now observe that all these relations are of the form $[u_{ia},u_{jb},u_{kc}]=0$, the precise assumptions being: 

(1a) $i,j,k$ distinct, and $a,b,c$ distinct.

(1b") $i=j$, and $a,b,c$ distinct.

($2^*$a) $i=j$, and $b=c$.

($2^*$b) $i=j$, and $a=b$.

(3) $a=c$.

Our claim is that, from these relations, we can deduce that we have $[u_{ia},u_{jb},u_{kc}]=0$, regardless of the indices. Indeed, let us look first at (1b"), ($2^*$a), ($2^*$b). These relations are of the same nature, involving the assumption $i=j$, and since by (3) the relation $[u_{ia},u_{jb},u_{kc}]=0$ holds as well for $i=j,a=c$, we can merge them. We conclude that the relations $[u_{ia},u_{jb},u_{kc}]=0$ hold, under the following assumptions:

(1a) $i,j,k$ distinct, and $a,b,c$ distinct.

($2^+$)  $i=j$.

(3) $a=c$.

Now by using the antipode, the relations ($2^+$), (3) tell us precisely that we have $[u_{ia},u_{jb},u_{kc}]=0$, whenever $i,j,k$ are not distinct, or when $a,b,c$ are not distinct. But this is exactly the complementary of the case covered by (1a), and we are done.
\end{proof}

Let us discuss now the remaining spheres, $S^{N-1}_{\mathbb C,\circ}\subset S^{N-1}_{\mathbb C,\#}$. We have here:

\begin{proposition}
The affine actions of $U_N^\circ,U_N^\#$ on $S^{N-1}_{\mathbb C,\circ},S^{N-1}_{\mathbb C,\#}$ are universal.
\end{proposition}

\begin{proof}
We use the same method as in the proof of Proposition 4.5. We first discuss the case of the sphere $S^{N-1}_{\mathbb C,\#}$. With $Z_i=\sum_uu_{ia}\otimes z_a$, we have:
$$Z_iZ_j^*=\sum_{ab}u_{ia}u_{jb}^*\otimes z_az_b^*$$

By using now the relations $z_az_b^*=z_b^*z_a$, this formula can be written as:
$$Z_iZ_j^*=\sum_{a<b}(u_{ia}u_{jb}^*+u_{ib}u_{ja}^*)\otimes z_az_b^*+\sum_au_{ia}u_{ja}^*\otimes z_az_a^*$$

By interchanging $i\leftrightarrow j$, we have as well the following formula:
$$Z_jZ_i^*=\sum_{a<b}(u_{ja}u_{ib}^*+u_{jb}u_{ia}^*)\otimes z_az_b^*+\sum_au_{ja}u_{ia}^*\otimes z_az_a^*$$

Now since by Lemma 4.3 (1) the variables on the right are independent, we conclude that the relations $Z_iZ_j^*=Z_jZ_i^*$ are equivalent to the following conditions:

(1) $u_{ja}u_{ib}^*-u_{ib}u_{ja}^*=u_{jb}u_{ia}^*-u_{ia}u_{jb}^*$.

(2) $u_{ja}u_{ia}^*=u_{ia}u_{ja}^*$.

Here we have dropped the assumption $a<b$ for the first relations, because by symmetry we have them for $a>b$ too, and these relations are automatic at $a=b$. By applying now the antipode to these relations, and then by relabelling, we succesively obtain:
$$u_{bi}u_{aj}^*-u_{aj}u_{bi}^*=u_{ai}u_{bj}^*-u_{bj}u_{ai}^*$$
$$u_{ja}u_{ib}^*-u_{ib}u_{ja}^*=u_{ia}u_{jb}^*-u_{jb}u_{ia}^*$$

Now by comparing with the original relations (1), we conclude that: 
$$u_{ja}u_{ib}^*-u_{ib}u_{ja}^*=u_{jb}u_{ia}^*-u_{ia}u_{jb}^*=0$$

In other words, the standard coordinates on a quantum group $G\curvearrowright S^{N-1}_{\mathbb C,\#}$ must satisfy $ab^*=ba^*$. Similarly, these coordinates must satisfy as well $a^*b=b^*a$. We conclude that we must have $G\subset U_N^\#$, and we are therefore done with the $S^{N-1}_{\mathbb C,\#}$ problem.

Regarding now the sphere $S^{N-1}_{\mathbb C,\circ}=S^{N-1}_{\mathbb C,\#}\cap S^{N-1}_{\mathbb C,**}$, our claim is that no new computation is needed. Consider indeed a quantum group $G\curvearrowright S^{N-1}_{\mathbb C,\circ}$. Since Lemma 4.3 (1) was valid over $S^{N-1}_{\mathbb C,\circ}$, the above computations apply, and we obtain $G\subset U_N^\#$. 

On the other hand, since Lemma 4.3 (2) was valid as well over $S^{N-1}_{\mathbb C,\circ}$, the computations in the proof of Proposition 4.5 apply as well, with the choice $\times=\emptyset$, and show that the standard coordinates on $G$ must satisfy the relations $abc=cba$.

In order to conclude, we use Lemma 4.4. We already know that the standard coordinates on $G$ satisfy the relations $abc=cba$, and from $G\subset U_N^\#\subset U_N^*$ we obtain that the relations $ab^*c=cb^*a$ are satisfied as well. Thus Lemma 4.4 applies, and gives $G\subset U_N^{**}$. We therefore conclude that we have $G\subset U_N^\#\cap U_N^{**}=U_N^\circ$, and we are done.
\end{proof}

We can now formulate our main result in this section, as follows:

\begin{theorem}
We have the following correspondence
$$\xymatrix@R=14mm@C=11mm{
S^{N-1}_\mathbb C\ar[r]&S^{N-1}_{\mathbb C,**}\ar[r]&S^{N-1}_{\mathbb C,*}\\
\mathbb TS^{N-1}_\mathbb R\ar[r]\ar[u]&S^{N-1}_{\mathbb C,\circ}\ar[u]\ar[r]&S^{N-1}_{\mathbb C,\#}\ar[u]}\quad
\xymatrix@R=10mm@C=10mm{\\ \longrightarrow}
\quad
\xymatrix@R=14.8mm@C=14mm{
U_N\ar[r]&U_N^{**}\ar[r]&U_N^*\\
\mathbb TO_N\ar[r]\ar[u]&U_N^\circ\ar[u]\ar[r]&U_N^\#\ar[u]}$$
between the $6$ spheres, and their affine quantum isometry groups.
\end{theorem}

\begin{proof}
The result for $S^{N-1}_\mathbb C$ is known since \cite{bhg}, the result for $\mathbb TS^{N-1}_\mathbb R$ is similar, with the argument in \cite{bhg} showing that we have indeed $G^+(\mathbb TS^{N-1}_\mathbb R)=G(\mathbb TS^{N-1}_\mathbb R)=\mathbb TO_N$, and the remaining results follow from Proposition 4.5 and Proposition 4.6 above.
\end{proof}

Summarizing, we have now some basic understanding of the 6 half-liberated spheres. There are, however, many questions left. A first series of questions concerns the ergodicity, uniqueness, and possible faithfulness of the $U_N^\times$-invariant integration on $S^{N-1}_\times$. A second series of questions concerns the construction of the Laplacian, and notably of its eigenvalues, and the possible Riemannian structure of $S^{N-1}_\times$. Finally, a third series of questions concerns the possible twisting of the above results. See \cite{ba2}, \cite{bgo}.

\section{Half-liberated manifolds}

We discuss now the extension of some of the results in sections 1-4, with the complex sphere $S^{N-1}_\mathbb C$ replaced by more general algebraic manifolds $X\subset S^{N-1}_\mathbb C$. There is in fact a lot of work to be done here, and we have so far only very partial results.

Generally speaking, the problem is that of constructing, under suitable assumptions on $X\subset S^{N-1}_\mathbb C$, a half-liberation diagram for it, as follows:
$$\xymatrix@R=15mm@C=15mm{
X\ar[r]&X^{**}\ar[r]&X^*\\
X^-\ar[r]\ar[u]&X^\circ\ar[u]\ar[r]&X^\#\ar[u]}$$

The starting point is Theorem 4.7 above. Forgetting that on the right we have quantum isometry groups, we can see that, besides the sphere $X=S^{N-1}_\mathbb C$ itself, we have as well the rescaled unitary group $X=\frac{1}{\sqrt{N}}\,U_N$ as example. Indeed, we have:

\begin{proposition}
We have embeddings as follows, given by $z_{ij}=\frac{1}{\sqrt{N}}\,u_{ij}$,
$$\xymatrix@R=16.2mm@C=14mm{
U_N\ar[r]&U_N^{**}\ar[r]&U_N^*\\
\mathbb TO_N\ar[r]\ar[u]&U_N^\circ\ar[u]\ar[r]&U_N^\#\ar[u]}\quad
\xymatrix@R=10mm@C=10mm{\\ \longrightarrow}
\quad
\xymatrix@R=14.6mm@C=10mm{
S^{N^2-1}_\mathbb C\ar[r]&S^{N^2-1}_{\mathbb C,**}\ar[r]&S^{N^2-1}_{\mathbb C,*}\\
\mathbb TS^{N^2-1}_\mathbb R\ar[r]\ar[u]&S^{N^2-1}_{\mathbb C,\circ}\ar[u]\ar[r]&S^{N^2-1}_{\mathbb C,\#}\ar[u]}$$
whose images are given by $\frac{1}{\sqrt{N}}\,U_N^\times=\frac{1}{\sqrt{N}}\,U_N^*\cap S^{N^2-1}_\times$.
\end{proposition}

\begin{proof}
Since the fundamental corepresentation $u=(u_{ij})$ of the quantum group $U_N^+$ is biunitary, we have $\sum_{ij}u_{ij}u_{ij}^*=\sum_{ij}u_{ij}^*u_{ij}=N$. Thus we have an embedding $U_N^+\subset S^{N^2-1}_{\mathbb C,+}$ given by $z_{ij}=\frac{1}{\sqrt{N}}\,u_{ij}$. Now since the quantum groups $U_N^\times$ in the statement appear by imposing to the standard coordinates $u_{ij}$ the same relations as those for the coordinates $z_{ij}$ on the corresponding spheres $S^{N^2-1}_\times$, we obtain $\frac{1}{\sqrt{N}}\,U_N^\times=\frac{1}{\sqrt{N}}\,U_N^*\cap S^{N^2-1}_\times$.
\end{proof}

The examples that we have so far, $X=S^{N-1}_\mathbb C$ and $X=\frac{1}{\sqrt{N}}\,U_N$, suggest an approach via ``lifting projective versions''. More precisely, given $X\subset S^{N-1}_\mathbb C$, consider its projective version $PX\subset P^N_\mathbb C$. Also, let $X^-=X\cap\mathbb TS^{N-1}_\mathbb R$, so that $PX^-=PX\cap P^N_\mathbb R$. The general idea is then to define $X^{**},X^*/X^\circ,X^\#$ as being the ``biggest'' submanifolds of the corresponding spheres, having $PX/PX^-$ as projective versions. 

In order for this idea to work, $X,X^-$ themselves must be the lifts to $S^{N-1}_\mathbb C,\mathbb TS^{N-1}_\mathbb R$ of their projective versions $PX,PX^-$. So, let us first recall that we have:

\begin{proposition}
For a subspace $X\subset S^{N-1}_\mathbb C$, the following are equivalent:
\begin{enumerate}
\item $X$ is the lift to $S^{N-1}_\mathbb C$ of its projective version $PX\subset P^N_\mathbb C$.

\item $X$ is invariant under the action of $\mathbb T$, given by $u\cdot z=(uz_i)_i$.
\end{enumerate}
In addition, in this case, $X^-=X\cap\mathbb TS^{N-1}_\mathbb R$ is the lift to $\mathbb TS^{N-1}_\mathbb R$ of $PX^-=PX\cap P^N_\mathbb R$.
\end{proposition}

\begin{proof}
Since the quotient map $\pi:S^{N-1}_\mathbb C\to P^N_\mathbb C$ satisfies $\pi(z)=\pi(z')\iff z'\in\mathbb Tz$, the lifting condition $X=\{x\in S^{N-1}_\mathbb C|\pi(x)\in PX\}$ is equivalent to the $\mathbb T$-invariance of $X$. 

Also, the quotient map $\sigma:\mathbb TS^{N-1}_\mathbb R\to P^N_\mathbb R$ satisfies as well $\sigma(z)=\sigma(z')\iff z'\in\mathbb Tz$, so the lifting condition $X^-=\{x\in\mathbb TS^{N-1}_\mathbb R|\sigma(x)\in PX^-\}$ is equivalent to the $\mathbb T$-invariance of $X^-$. But if $X$ is $\mathbb T$-invariant, then so is $X^-$, and this gives the last assertion. 
\end{proof}

The other problem is that the general noncommutative manifolds $Z\subset S^{N-1}_{\mathbb C,+}$ have in fact two projective versions, one given by $p_{ij}=z_iz_j^*$, and the other one given by $q_{ij}=z_j^*z_i$. In order to deal with this issue, best is to assume that all our manifolds $Z$ are ``conjugation-stable'', in the sense that $C(Z)$ has an anti-automorphism given by $z_i\to z_i^*$. 

Observe that for the manifold $X\subset S^{N-1}_\mathbb C$ itself, the stability under conjugation, which comes from an action of $\mathbb Z_2$, can be combined with the stability under the action of $\mathbb T$, coming from Proposition 5.2 above. In this case, we say that $X$ is $O_2$-invariant.

We can now formulate our half-liberation construction, as follows:

\begin{definition}
If $X\subset S^{N-1}_\mathbb C$ is $O_2$-invariant, we set $X^-=X\cap\mathbb TS^{N-1}_\mathbb R$, and we define
$$\xymatrix@R=16mm@C=15mm{
X\ar[r]&X^{**}\ar[r]&X^*\\
X^-\ar[r]\ar[u]&X^\circ\ar[u]\ar[r]&X^\#\ar[u]}$$
by the fact that $X^{**},X^*/X^\circ,X^\#$ are the conjugation-stable lifts of $PX/PX^-$.
\end{definition}

As a basic example, for the sphere $X=S^{N-1}_\mathbb C$ we have $PX=P^N_\mathbb C$, $PX^-=P^N_\mathbb R$, the lifting problem is trivial, and we obtain the $6$ half-liberated spheres themselves. 

Observe also that, due to our $\mathbb T$-invariance assumption on $X$, all the 6 spaces appearing in the above diagram are the lifts of their projective versions $PX,PX^-$. 

In general, the fact that the above lifts exist indeed follows by dividing the corresponding algebras by suitable ideals. Let us record a more precise result here:

\begin{proposition}
The spaces $X^\times$ appear via  $C(X^\times)=C(S^{N-1}_\times)/<I,J>$, where
$$I/J=
\begin{cases}
\ker[C(P^N_\mathbb C)\to C(PX)]&{\rm at}\ \times=**,*\\
\ker[C(P^N_\mathbb R)\to C(PX^-)]&{\rm at}\ \times=\circ,\#
\end{cases}$$
regarded as linear subspaces of $C(S^{N-1}_\times)$, via the embeddings $p_{ij}= z_iz_j^*/q_{ij}=z_j^*z_i$.
\end{proposition}

\begin{proof}
At the algebra level, the lifts at $\times=**,*$ and at $\times=\circ,\#$ in Definition 5.3 above are by definition the universal solutions to the following problems:
$$\xymatrix@R=15mm@C=15mm{
C(P^N_\mathbb C)\ar[r]\ar[d]&C(S^{N-1}_\times)\ar@.[d]\\
C(PX)\ar@.[r]&C(X^\times)
}\qquad\qquad
\xymatrix@R=15mm@C=15mm{
C(P^N_\mathbb R)\ar[r]\ar[d]&C(S^{N-1}_\times)\ar@.[d]\\
C(PX^-)\ar@.[r]&C(X^\times)
}$$

But the solutions to these problems are given by formula in the statement.
\end{proof}

As an illustration, consider the space $X=\frac{1}{\sqrt{N}}\,\mathbb T^N$ formed by the points  $z\in S^{N-1}_\mathbb C$ satisfying $|z_i|=\frac{1}{\sqrt{N}}$ for any $i$. Here we have $X^-=\frac{1}{\sqrt{N}}\,\mathbb T\mathbb Z_2^N$, and the result is:

\begin{proposition}
We have the following (rescaled) half-liberation diagram,
$$\xymatrix@R=16mm@C=15mm{
\mathbb T^N\ar[r]&\widehat{\mathbb Z^{\diamond\diamond N}}\ar[r]&\widehat{\mathbb Z^{\diamond N}}\\
\mathbb T\mathbb Z_2^N\ar[r]\ar[u]&\widehat{\mathbb Z^{\circ N}}\ar[u]\ar[r]&\widehat{\mathbb Z^{\# N}}\ar[u]}$$
where $\diamond,\diamond\diamond,\#,\circ$ are the group-theoretic analogues of the operations $*,**,\#,\circ$.
\end{proposition}

\begin{proof}
Observe first that given a discrete group $\Gamma=<g_1,\ldots,g_N>$, we have an embedding $\widehat{\Gamma}\subset S^{N-1}_{\mathbb C,+}$, given by $z_i=\frac{1}{\sqrt{N}}\,g_i$, and that over $P\widehat{\Gamma}$ we have $p_{ii}=q_{ii}=\frac{1}{N}$.

In our case, we deduce from $p_{ii}=q_{ii}=\frac{1}{N}$ that we have $z_iz_i^*=z_i^*z_i=\frac{1}{N}$, over the various lifts. Thus the rescaled lifts are group duals, given by $\widehat{\mathbb Z^{\times N}}=\widehat{F_N}\cap\sqrt{N}S^{N^2-1}_\times$. But this gives the result, with the $\diamond,\#$ constructions obtained respectively by imposing the conditions $ab^{-1}c=cb^{-1}a,ab^{-1}=ba^{-1}$ to the standard generators of $F_N$, and with the $\diamond\diamond,\circ$ constructions being obtained by further imposing the relations $abc=cba$.
\end{proof}

Let us check the fact that the example $X=\frac{1}{\sqrt{N}}\,U_N$ is covered as well:

\begin{proposition}
For the rescaled unitary group $X=\frac{1}{\sqrt{N}}\,U_N$, the abstract half-liberation construction produces the (rescaled) $6$ half-liberated quantum groups $U_N^\times$.
\end{proposition}

\begin{proof}
We first discuss the lifting problem for $PU_N\subset P^{N^2}_\mathbb C$. If we denote by $U_N^{\times\times}$ the rescaling of the lift inside $S^{N-1}_{\mathbb C,\times}$, we have the following series of implications:
\begin{eqnarray*}
&&uu^*=u^*u=u^t\bar{u}=\bar{u}u^t=1,\ {\rm over}\ U_N\\
&\implies&\sum_ku_{ik}\bar{u}_{jk}=\sum_k\bar{u}_{ki}u_{kj}=\sum_ku_{ki}\bar{u}_{kj}=\sum_k\bar{u}_{ik}u_{jk}=\delta_{ij},\ {\rm over}\ U_N\\
&\implies&\sum_kp_{ik,jk}=\sum_kq_{ki,kj}=\sum_kp_{ki,kj}=\sum_kq_{ik,jk}=\delta_{ij},\ {\rm over}\ N\cdot PU_N\\
&\implies&\sum_ku_{ik}u_{jk}^*=\sum_ku_{ki}^*u_{kj}=\sum_ku_{ki}u_{kj}^*=\sum_ku_{ik}^*u_{jk}=\delta_{ij},\ {\rm over}\ U_N^{\times\times}\\
&\implies&uu^*=u^*u=u^t\bar{u}=\bar{u}u^t=1,\ {\rm over}\ U_N^{\times\times}
\end{eqnarray*}

Thus we have an inclusion $U_N^{\times\times}\subset U_N^+$. But since $U_N^\times$ is by definition given by $U_N^\times=U_N^+\cap\sqrt{N}S^{N^2-1}_{\mathbb C,\times}$, we conclude that we have $U_N^{\times\times}=U_N^\times$, as desired.

For the lifting problem for $PO_N\subset P^{N^2}_\mathbb R$ we can use the same proof, because the above middle relations, over $N\cdot PU_N$, hold over $N\cdot PO_N$ as well.
\end{proof}

Let us work out as well a ``discrete'' analogue of Proposition 5.6. Consider the group $K_N\subset U_N$ of matrices which are monomial, in the sense that each row and each column has exactly one nonzero entry. Its free version $K_N^+\subset U_N^+$ is then defined via the relations $ab^*=a^*b=0$, for any $a\neq b$ on the same row or column of $u$. See \cite{ba1}. 

With the notations, we have the following result:

\begin{proposition}
With $X=\frac{1}{\sqrt{N}}\,K_N$ we obtain the following diagram,
$$\xymatrix@R=15mm@C=15mm{
K_N\ar[r]&K_N^{**}\ar[r]&K_N^*\\
\mathbb TH_N\ar[r]\ar[u]&K_N^\circ\ar[u]\ar[r]&K_N^\#\ar[u]}$$
rescaled by $\frac{1}{\sqrt{N}}$, where on the right we have the quantum groups $K_N^\#=K_N^+\cap U_N^\#$.
\end{proposition}

\begin{proof}
The space $K_N^-=\sqrt{N}X^-$ is given by $K_N^-=K_N\cap\sqrt{N}\mathbb TS^{N^2-1}_\mathbb R$, and we therefore obtain $K_N^-=K_N\cap\mathbb TO_N=\mathbb TH_N$, where $H_N\subset O_N$ is the hyperoctahedral group.

Let us first compute the various lifts of $PK_N$. We already know from Proposition 5.6 that these lifts satisfy $K_N^{\times\times}\subset U_N^\times$. Also, for $j\neq k$ we have:
\begin{eqnarray*}
&&u_{ij}\bar{u}_{ik}=\bar{u}_{ij}u_{ik}=u_{ji}\bar{u}_{ki}=\bar{u}_{ki}u_{ki}=0,\ {\rm over}\ K_N\\
&\implies&p_{ij,ik}=q_{ij,ik}=p_{ji,ki}=q_{ji,ki}=0,\ {\rm over}\ PK_N\\
&\implies&u_{ij}u_{ik}^*=u_{ij}^*u_{ik}=u_{ji}u_{ki}^*=u_{ji}^*u_{ki}=0,\ {\rm over}\ K_N^{\times\times}
\end{eqnarray*}

We conclude that the lifts appear inside $U_N^\times$ via the relations $ab^*=a^*b=0$, for any $a\neq b$ on the same row or column of $u$. Thus we have $K_N^{\times\times}\subset K_N^+$, and we are done. 

The lifting problem for $P\mathbb TH_N=PH_N$ is similar, by using the same computation.
\end{proof}

Summarizing, the half-liberation operation that we constructed leads to quite natural objects, in all the cases investigated so far. In particular, we can now formulate:

\begin{theorem}
For the half-liberations of the sphere $X=S^{N-1}_\mathbb C$ we have 
$$G^+(X^\times)=G(X)^\times$$
with the quantum isometry groups being taken in an affine sense.
\end{theorem}

\begin{proof}
This is just a reformulation of the results that we proved before, in Theorem 4.7, by using the abstract half-liberation formalism developed above.
\end{proof}

Observe that the formula established above could be thought of as being related to the various rigidity results of type $G^+(X)=G(X)$, from \cite{bhg}, \cite{gjo}. 

In general, it is quite unclear what exact assumptions on $X\subset S^{N-1}_\mathbb C$ could lead to such results. This is an interesting question, that we would like to raise here.

\section{Real versions}

In this section we discuss a number of more specialized results, concerning the real versions of our half-liberations, obtained by imposing the conditions $z_i=z_i^*$ to the standard coordinates. First, we have the following elementary result:

\begin{proposition}
The real versions of the half-liberations $X^\times$ are
$$\xymatrix@R=16mm@C=15mm{
X_\mathbb R\ar[r]&X_\mathbb R^*\ar[r]&X_\mathbb R^*\\
X_\mathbb R\ar[r]\ar[u]&X_\mathbb R\ar[u]\ar[r]&X_\mathbb R\ar[u]}$$
where $X_\mathbb R=X\cap S^{N-1}_\mathbb R$, and $X_\mathbb R^*=X^{**}\cap S^{N-1}_{\mathbb R,+}$.
\end{proposition}

\begin{proof}
This follows indeed from the last assertion in Proposition 1.4 above, which tells us that taking the real versions amounts in intersecting with $S^{N-1}_{\mathbb R,*}/S^{N-1}_\mathbb R$.
\end{proof}

We can axiomatize the construction $X\to X_\mathbb R^*$, as follows:

\begin{proposition}
Given an $O_2$-invariant closed subset $X\subset S^{N-1}_\mathbb C$, the closed subset $X_\mathbb R^*\subset S^{N-1}_{\mathbb R,*}$ appears by lifting the projective version $PX\subset P^N_\mathbb C$.
\end{proposition}

\begin{proof}
This follows from Proposition 5.4 above, because the variables $p_{ij}=z_iz_j^*$ and $q_{ji}=z_i^*z_j$ being now equal, the conjugation-stable lift becomes a plain lift.
\end{proof}

At the level of examples now, we have the following result:

\begin{proposition}
We have the following plain/rescaled real half-liberations,
$$(S^{N-1}_\mathbb R)^*=S^{N-1}_{\mathbb R,*}\quad/\quad (\mathbb Z_2^N)^*=\widehat{\mathbb Z_2^{\diamond N}},\quad (O_N)^*=O_N^*,\quad (H_N)^*=H_N^*$$
coming respectively from the complex manifolds $S^{N-1}_\mathbb C/\mathbb T^N,U_N,K_N$.
\end{proposition}

\begin{proof}
The first assertion is clear from the comments made after Definition 5.3. 

Regarding the second assertion, we can use here Proposition 5.5, which tells us that the rescaled real-half liberation in question is:
$$\widehat{\mathbb Z^{\diamond\diamond N}}\cap S^{N-1}_{\mathbb R,+}=\widehat{\mathbb Z_2^{\diamond\diamond N}}=\widehat{\mathbb Z_2^{\diamond N}}$$ 

Finally, the last two assertions are clear from Proposition 5.6 and Proposition 5.7.
\end{proof}

We have as well the following matrix model result, obtained by using the doubling operation $X\to|X|$, constructed in section 2 above:

\begin{proposition}
If $X\subset S^{N-1}_\mathbb C$ is $O_2$-invariant, then $|X|\subset X_\mathbb R^*$.
\end{proposition}

\begin{proof}
We recall from Proposition 2.2 that we have $|S^{N-1}_\mathbb C|\subset S^{N-1}_{\mathbb R,*}$, and so the result holds for $X=S^{N-1}_\mathbb C$ itself. In general now, observe that we have:
$$z_i'(z_j')^*=\begin{pmatrix}0&z_i\\ \bar{z}_i&0\end{pmatrix}
\begin{pmatrix}0&z_j\\ \bar{z}_j&0\end{pmatrix}
=\begin{pmatrix}z_i\bar{z}_j&0\\ 0&\bar{z}_iz_j\end{pmatrix}$$

Now since $X\subset S^{N-1}_\mathbb C$ is $O_2$-invariant, $z\to\bar{z}$ induces an automorphism of $C(X)$, and so an automorphism of $C(PX)$. We can therefore ``cut'' the lower part of the above matrix, and we obtain $P|X|=PX$. Thus $|X|$ lifts $PX$, and so $|X|\subset X_\mathbb R^*$, as desired.
\end{proof}

Regarding now the quantum isometry groups, the fact that we have $G^+(S^{N-1}_{\mathbb R,*})=O_N^*$ was already known from \cite{ba2}. We can improve now this result. We use:

\begin{definition}
A closed subspace $X\subset Y$ is called $k$-saturated when the dimension of $span(z_{i_1}^{e_1}\ldots z_{i_k}^{e_k})$ does not decrease via $C(Y)\to C(X)$, for any $e_1,\ldots,e_k\in\{1,*\}$.  
\end{definition}

Observe that the 1-saturation of $X\subset S^{N-1}_\mathbb C$, which is equivalent to the fact that the coordinates $z_1,\ldots,z_N\in C(X)$ are linearly independent, is needed in order to define the affine quantum isometry group $G^+(X)$, as a closed subgroup of $U_N^+$. See \cite{gos}.

The 2-saturation condition is a familiar one as well, because for a subset $X\subset S^{N-1}_\mathbb C$, this condition implies that we have $G^+(X)=G(X)$, as shown in \cite{bhg}.

We have the following result, regarding the 3-saturated sets: 

\begin{theorem}
We have the ``half-classical rigidity'' formula
$$G^+(X_\mathbb R^*)\subset O_N^*$$
provided that $X\subset S^{N-1}_\mathbb C$ is $O_2$-invariant and $3$-saturated.
\end{theorem}

\begin{proof}
Assuming that $X\subset S^{N-1}_\mathbb C$ is 3-saturated, the doubling $|X|\subset S^{N-1}_{\mathbb R,*}$ is $3$-saturated as well, and we conclude that the half-liberation $X_\mathbb R^*\subset S^{N-1}_{\mathbb R,*}$ is $3$-saturated too.

Thus, the variables $\{z_az_b^\times z_c|a\leq c\}$ with $\times=1,*$ are linearly independent, and so the method in the proof of Proposition 4.5 applies, and gives the result.
\end{proof}

Observe the similarity between the above result and Theorem 5.8. 

As a conclusion, our various results suggest that a certain analogue of the rigidity result in \cite{gjo} should hold in real and complex half-liberated affine geometry. Finding such a general result, however, looks like a quite difficult question.

\end{document}